\documentclass[leqno,a4paper,12pt]{amsart}
\usepackage{amsmath,amsthm,amssymb}
\usepackage[utf8]{inputenc}
\usepackage[T1]{fontenc}

\usepackage{enumerate}
\usepackage{bbm}
\usepackage{todonotes}
\usepackage{mathrsfs}
\usepackage{mathabx}
\usepackage{hyperref}
\usepackage{nicefrac}
\usepackage{tikz}
\usetikzlibrary{matrix}

\usepackage{url}
\newtheorem{Th}{Theorem}[section]
\newtheorem*{thma}{Theorem A}
\newtheorem*{thmb}{Theorem B}
\newtheorem{Prop}[Th]{Proposition}
\newtheorem{Lemma}[Th]{Lemma}
\newtheorem{Cor}[Th]{Corollary}

\theoremstyle{definition}
\newtheorem{Remark}[Th]{Remark}

\newcommand{\beq}{\begin{equation}}
\newcommand{\eeq}{\end{equation}}

\def\scalar(#1,#2){(#1\mid#2)}

\newcommand{\R}{{\mathbb{R}}}
\newcommand{\T}{{\mathbb{T}}}
\newcommand{\C}{{\mathbb{C}}}
\newcommand{\Z}{{\mathbb{Z}}}
\newcommand{\N}{{\mathbb{N}}}

\newcommand{\PP}{{\mathbb{P}}}

\newcommand{\vep}{\varepsilon}

\newcommand{\mob}{\boldsymbol{\mu}}

\begin{document}

\title{Prime number theorem for regular Toeplitz subshifts}
\author[K.~Fr\k{a}czek]{Krzysztof Fr\k{a}czek}
\address{Faculty of Mathematics and Computer Science, Nicolaus
Copernicus University, ul. Chopina 12/18, 87-100 Toru\'n, Poland}
\email{fraczek@mat.umk.pl}

\author[A.~Kanigowski]{Adam Kanigowski}
\address{Department of Mathematics, University of Maryland, 4176 Campus Drive, William E.\ Kirwan Hall,
College Park, MD 20742-4015, USA}
\email{akanigow@umd.edu}

\author[M.~Lema\'nczyk]{Mariusz Lema\'nczyk}
\thanks{Research supported by Narodowe Centrum Nauki grant 2019/33/B/ST1/00364.}
\address{Faculty of Mathematics and Computer Science, Nicolaus
Copernicus University, ul. Chopina 12/18, 87-100 Toru\'n, Poland}
\email{mlem@mat.umk.pl}

\subjclass[2000]{37B10, 37A45, 11N05, 11N13}
\keywords{prime number theorem, Toeplitz systems, almost prime numbers, polynomial ergodic theorems}

\begin{abstract}
We prove that neither a prime  nor {an $l$-almost prime} number  theorem hold in the class of regular  Toeplitz subshifts. But, {when a quantitative strengthening of the regularity with respect to the periodic structure involving Euler's totient function is assumed}, then the two theorems hold.
\end{abstract}

\maketitle

\section{Introduction} Given a topological dynamical system $(X,T)$, where $T$ is a homeomorphism of a compact metric space $X$, one says that a prime number theorem (PNT) holds {for $(X,T)$} if the limit
\begin{equation}\label{pnt31}
\lim_{N\to\infty}\frac1{\pi(N)}\sum_{p<N}f(T^px)
\end{equation}
($p$ stands always for a prime number) exists for each $x\in X$, an arbitrary $f\in C(X)$ and $\pi(N)$ denotes the number of primes up to $N$. {Then}, via the Riesz theorem, for all $f\in C(X)$, we have
\begin{equation}\label{pnt31a}
\lim_{N\to\infty}\frac1{\pi(N)}\sum_{p<N}f(T^px)=\int_X f\,d\nu_x
\end{equation}
for a {Borel probability measure} $\nu_x$ on $X$, where $\nu_x$ depends only on $x\in X$.

Let us first consider the cyclic case $X=\Z/k\Z$ and $Tx=x+1$. Fix $x\in X$ and notice that~\eqref{pnt31} indeed holds by the classical prime number theorem in arithmetic progressions, where
$\nu_x$ is the uniform probability measure on the ``coset'' $\{a<k:\: (a,k)=1\}+x$. Hence, a PNT holds in cyclic (and therefore also in finite) systems.

{Consider now the procyclic case, that is, assume we are given an odometer system $(H,T)$ with
\[
H={\rm liminv}_{t\to\infty}\,\Z/n_t\Z, \;Tx=x+(1,1,\ldots)
\]
(here $n_t|n_{t+1}$ for $t\geq 0$). In this case, a PNT still holds. Indeed, the space $H$ has a sequence of natural partitions  $D^t=(D^t_0,\ldots,D^t_{n_t-1})$, $t\geq 0$, consisting of clopen sets and such that $TD^t_i=D^t_{i+1\; {\rm mod}~n_t}$. It follows that  the sets $D^t_i$, $i\leq n_t-1$, have the same diameter which goes to $0$ as $t\to \infty$. Moreover, it is not hard to see that each character of the group $H$ is constant on the levels of the towers $D^t$ for $t$ sufficiently large.
Hence, each $f\in C(H)$ can be approximated {\bf uniformly} by functions which are constant on the levels of the towers $D^t$  and a PNT holds because it does in the finite case.}

{Our main results concern prime number theorems for extensions of odometers. Recall that odometers  are zero entropy topological systems which are minimal (all $T$-orbits are dense) and uniquely ergodic (there is only one $T$-invariant measure - Haar measure in this case). Before we describe our results, let us discuss a PNT in the class of uniquely ergodic systems. First, recall that for all such systems~\eqref{pnt31} holds} a.e.\ with respect to the unique invariant measure \cite{Bo}, \cite{Wi}. On the other hand, one can easily construct a counterexample to the validity of~\eqref{pnt31} for {\bf all} $x\in X$.  Indeed, denote  by $\PP$  the set of prime numbers and consider the left shift $S$ on $\{0,1\}^{\Z}$ and the subshift  $(X_{{\mathbf 1}_{\PP\cup-\PP}}, S)$~ obtained by the orbit closure of the characteristic function ${\mathbf 1}_{\PP\cup-\PP}$ of the ``symmetrized'' primes. It  has a unique invariant measure {of zero entropy} (which is the Dirac measure at the fixed point $\ldots0.00\ldots$) and a PNT fails in  it (see e.g.\ \cite{Fe-Ku-Le} for details). Now, this particular uniquely ergodic model of the one-point system implies paradoxically that
each ergodic dynamical system has a uniquely ergodic model $(X,T)$~\footnote{Recall that the Jewett-Kreiger theorem says the following: Suppose $(Z,\kappa,R)$ is an ergodic measure-theoretic dynamical system. Then there exists a uniquely ergodic (even strictly ergodic, that is, additionally minimal) topological system $(Y,S)$ with the unique invariant measure $\nu$ such that $(Z,\kappa,R)$ and $(Y,\nu,S)$ are {\bf measure-theoretically} isomorphic.} in which a PNT does not hold. To see this,
take any uniquely ergodic model $(Y,\nu,R)$ of the given measure-theoretic dynamical system. Since the one-point system is (Furstenberg) disjoint with any other system, the product system
$(X_{{\mathbf1}_{\PP\cup-\PP}}\times Y, S\times R)$ is still uniquely ergodic, with the unique invariant measure $\delta_{\ldots0.00\ldots}\otimes\nu$. It is not hard to see that the product  system is still measure-theoretically isomorphic to the original system. Since the new system has $(X_{{\mathbf 1}_{\PP\cup-\PP}},S)$ as
 its {\bf topological} factor, a PNT does not hold in
$(X_{{\mathbf1}_{\PP\cup-\PP}}\times Y, S\times R)$.\footnote{To illustrate this, consider an irrational rotation $R_\alpha$ on $\T$ for which a PNT holds because of Vinogradov's theorem (prime ``orbits'' are equidistributed). However, our observation shows that there is a uniquely ergodic model of $R_\alpha$ in which the eigenfunctions are still continuous but a PNT fails, that is, some of the prime ``orbits'' are not equidistributed.}  Hence, if we think about a necessary condition for a PNT to hold, it looks reasonable to add the minimality assumption to avoid a problem of ,,exotic'' orbits on which PNT does not hold (we also recall that a uniquely ergodic system has a unique subsystem which is strictly ergodic).
{However, in this class one can still produce counterexamples to a PNT, see \cite{Pa} for the first symbolic counterexamples (although their entropy is not determined in \cite{Pa}), or  \cite{Ka-Le-Ra} for non-symbolic counterexamples.} On the other hand, we have quite a few classes in which a PNT holds, including systems of algebraic origin \cite{Gr-Ta}, \cite{Vi}, symbolic systems \cite{Bo1}, \cite{Gr}, \cite{Ma-Ri}, \cite{Mu} or recently \cite{Ka-Le-Ra} in the category of smooth systems, where a PNT has been proved in the class of analytic Anzai skew products. Finding a sufficient {\bf dynamical} condition for a PNT to hold, postulated a few years ago by P.~Sarnak \cite{Sa-talk} seems to be an important and difficult task in dynamics, however we rather expect the following:

\medskip
\noindent\textbf{Working Conjecture:}
{\em Each ergodic and aperiodic\footnote{{The set of periodic points has measure zero.}} measure-theoretic dynamical system has a strictly ergodic model in which a PNT fails.}

\medskip

{If true}, this makes Sarnak's postulate even harder to realize. The present paper should be viewed as introductory steps in trying to understand the conjecture.

A PNT can be reformulated as the existence of a limit of $\frac1N\sum_{n<N}f(T^nx)\Lambda(n)$, where $\Lambda$ stands for the von Mangoldt function: $\Lambda(p^\ell)=\log p$ for $\ell\geq1$ and 0 otherwise.
{Proving dynamical prime number theorems for zero entropy systems is closely related to  Sarnak's  M\"obius disjointness conjecture \cite{Sa}:}
\begin{equation}\label{sar1}
\lim_{N\to\infty}\frac1N\sum_{n<N}f(T^nx)\mob(n)=0\end{equation}
for each $x\in X$, $f\in C(X)$ in each {\bf zero entropy} dynamical system $(X,T)$ ($\mob$ stands for the M\"obius function: $\mob(1)=1$, $\mob(p_1\cdot\ldots\cdot p_k)=(-1)^k$ for different primes $p_1,\ldots, p_k$, and $\mob(n)=0$ for the remaining $n\in\N$). Here, the {class of systems} for which we expect the positive answer is precisely defined. {In fact, in quite a few cases (see \cite{Bo1}, \cite{Bo2}, \cite{Fe-Ma}-\cite{Gr-Ta}, \cite{Ma-Ri} and \cite{Mu}) one can observe the} following principle: once we can prove Sarnak's conjecture for $(X,T)$ with a ``sufficient'' speed of convergence to zero in~\eqref{sar1} then a PNT holds in $(X,T)$.

With all the above in mind we come {back to extensions of odometers that we intend to study.} We stay in the {\bf zero entropy} category of systems and we assume {\bf minimality}. Further, we assume that the systems are {\bf almost 1-1 extensions} of odometers.\footnote{If $(H,T)$ is a factor of $(X,S)$ via $\pi:X\to H$, then $(X,S)$ is called  an {\em almost 1-1 extension} of $(H,T)$ if there is a point $h\in H$ such that $|\pi^{-1}(h)|=1$; in fact, in this case the set of points with singleton fibers is $G_\delta$ and dense.} We also assume that our systems are {\bf symbolic}.\footnote{We recall that each zero entropy system has an extension which is symbolic \cite{Bo-Fi-Fi}, and clearly if a PNT holds for a system, it does for a factor.} All these natural assumptions determine however a very precise class of topological systems, namely Toeplitz subshifts $(X_x, S)$, where $x$ is a Toeplitz sequence over a finite alphabet $\mathcal{A}$, {see Section~7 in Downarowicz's survey \cite{Do}.} That is, {$x\in \mathcal{A}^{\Z}$ has the property that
for every $a\in \Z$ there is $\ell\in\N$ such that $x(a)=x({a+k\ell})$ for each $k\in\Z$, and $X_x$ is the set of all $y\in\mathcal{A}^{\Z}$ with the property that all subblocks of $y$ also appear in $x$.}
One shows then that there is a sequence $n_t|n_{t+1}$ such that if ${\rm Per}_{n_t}(x):=\{a\in\Z:\: x(a)=x(a+kn_t)\text{ for each }k\in\Z\}$ then \begin{equation}\label{max}\bigcup_{t\geq0}{\rm Per}_{n_t}(x)=\Z.\end{equation} Moreover, there is a natural continuous factor map $\pi:X_x\to H$, where $H$ stands for the odometer determined by $(n_t)$.
In fact, we will restrict our attention to so called {\em regular} Toeplitz subshifts,
whose formal definition is that the density of $\bigcup_{t=0}^M{\rm Per}_{n_t}(x)$ {goes} to $1$. Regular Toeplitz subshifts are zero entropy strictly ergodic systems, and measure-theoretically isomorphic to the rotation given by {their} maximal equicontinuous factors. Although in \cite{Do} there are four other equivalent conditions for regularity (see Theorem 13.1 in \cite{Do}), we will choose a different path. {Since} $\pi:X_x\to H$ is a continuous and equivariant  surjection,
\[E^t:=\pi^{-1}(D^t)=(E^t_0,\ldots, E^t_{n_t-1})\text{ with }E^t_j=\pi^{-1}(D^t_j)\]
is an $S$-tower of height $n_t$ whose levels are closed (hence clopen). By the minimality of $(X_x,S)$ there is a unique tower with clopen levels and of fixed height.
{Let us consider a metric on $\mathcal{A}^\Z$ inducing the product topology given by}
\[d(x,y)=2^{-\inf\{|n|:x(n)\neq y(n)\}}.\]
The diameters of the levels of towers $E^t$ {\bf do not} converge to zero, unless $x$ is periodic. Moreover, the diameters of different levels are in general different as the shift $S$ is not an isometry.
Let us consider the diameter of the tower $E^t$ given by:
\[\delta(E^t):=\sum_{0\leq j<n_t}\operatorname{diam}(E_j^t).\]
It is not hard to see (see Appendix~\ref{app:?}) that the regularity of a Toeplitz sequence is equivalent to
\begin{equation}\label{weyl1}
\lim_{t\to\infty}\frac{\delta(E^t)}{n_t}=0.
\end{equation}
It is also not hard to see that this property does not depend on the choice of $(n_t)$ satisfying~\eqref{max}. We recall that the M\"obius disjointness of subshifts given by regular Toeplitz sequences has been proved in \cite{Ab-Ka-Le}. Here are two first results of the paper proved in Section~\ref{s:Sec1} and Section~\ref{s:Sec3}, respectively:

\begin{thma}
A PNT does not hold in {the class of minimal almost $1-1$ symbolic extensions of odometers satisfying~\eqref{weyl1}.} That is, a PNT need not hold in a strictly ergodic subshift determined by a {\bf regular} Toeplitz sequence.
\end{thma}

\begin{thmb}
A PNT holds in the class of minimal almost $1-1$ symbolic extensions of odometers in which \eqref{weyl1} holds with a speed
\begin{equation}\label{weyl2}
\lim_{t\to\infty}\frac{\delta(E^t)}{\varphi(n_t)}=0,
\end{equation}
where $\varphi$ denotes the Euler totient function.
\end{thmb}

As for all Toeplitz dynamical systems constructed in the proof of Theorem~A, we have
\[
0<\liminf_{t\to\infty}\frac{\delta(E^t)}{\varphi(n_t)}\leq \limsup_{t\to\infty}\frac{\delta(E^t)}{\varphi(n_t)}<+\infty,
\]
{which shows that the condition \eqref{weyl2}} in Theorem B is optimal to
have a PNT. 
{The systems in Theorem~B are strictly ergodic and since they all have non-trivial cyclic factors, the measures  $\nu_y$, $y\in X_x$,~in~\eqref{pnt31a} are never $S$-invariant.\footnote{To be compared with the case of Sturmian systems, see Theorem~B.1, in which $\nu_y$, $y\in X_{\alpha,\beta}$, are equal to the unique $S$-invariant measure.}}

\medskip

{
We then turn our attention to an $l$-almost prime number theorem (P$_l$NT) which  is much less explored than the PNT case and which, for the first time in dynamics, is studied in \cite{Ka-Le-Ra1} (for some smooth Anzai skew products). Recall that for any $l\geq 1$ a natural number is called an $l$-almost prime if it is a product of $l$ primes. We denote the set of $l$-almost prime numbers by $\mathbb{P}_l$. By $\mathbb{P}_l^N$ we denote the set of $l$-almost prime numbers $\leq N$ and we let $\pi_l(N)$ stand for the cardinality of $\mathbb{P}_l^N$. A classical result of Landau asserts that
\begin{equation}\label{land}
\lim_{N\to\infty}\frac{\pi_l(N)}{\frac{N}{\log N}\frac{(\log\log N)^{l-1}}{(l-1)!}}=1,
\end{equation}
see $\S$ 56 in \cite{La}.}

{Analogously to the PNT, we say that a topological dynamical system $(X,T)$ satisfies a P$_l$NT if the limit
\[
\lim_{N\to\infty}\frac1{\pi_l(N)}\sum_{n\in \mathbb{P}_l^N}f(T^nx)
\]
exists for each $x\in X$ and each $f\in C(X)$.}

 {In Section~\ref{s:SecPlNT} and Section~\ref{s:Sec4}  we provide sketches of proofs of the exact analogues of Theorems~A and~B for a P$_l$NT for regular Toeplitz subshifts.}

In Section~\ref{sec:nonconv} we prove a new {polynomial ergodic theorem}:
\[\lim_{N\to\infty}\frac1N\sum_{n\leq N}f(S^{P(n)}x)\ \text{ exists} \]
for monic polynomials $P$ with positive integer coefficients for all symbolic minimal almost 1-1 extensions of odometers with a modified condition~\eqref{weyl2}. In Section~\ref{s:kwadrat} we provide a regular Toeplitz subshift which does not satisfy the {polynomial ergodic theorem} for squares but it satisfies a PNT. {We refer again to \cite{Pa} for the first examples of strictly ergodic systems (of low complexity), where the Birkhoff ergodic averages along squares do not converge.}

While Theorem~A confirms {the Working Conjecture} for a subclass of odometers, we have been unable to confirm it for the whole class of  odometers. Confirming Working Conjecture for the class of  automorphisms with discrete spectrum seems to be the first step toward a possible general statement. In Appendix~\ref{s:dosturmu}, we provide a simple argument {showing
that a PNT holds for all symbolic models of irrational rotations given by Sturmian sequences.} The Sturmian systems are strictly ergodic and are almost $1-1$ extensions of irrational rotations.

\section{Regular Toeplitz subshifts which do not satisfy PNT (proof of Theorem~A)}\label{s:Sec1}

For all $K,n\in \N$ and $a\in\Z$ let
\[\pi(K;n,a)=\{1\leq p\leq K:p\in\mathbb{P}, p=a\!\mod n\}.\]

\begin{Th}[{PNT in arithmetic progressions, see \cite{Sel}}]\label{thm:Diri}
For any natural $n$ and any integer $a$ with $(a,n)=1$ we have
\[\lim_{K\to\infty}\frac{\pi(K;n,a)}{{\pi(K)}}=\frac{1}{\varphi(n)}.\]
\end{Th}

We construct a Toeplitz sequence $x\in\{0,1\}^{\Z}$
with the period structure $(n_t)$:
\begin{equation}\label{t1}
n_{t+1}=k_{t+1}n_t,\;(k_{t+1},n_t)=1\end{equation}
for each $t\geq1$. We will show that for this $x$:\
\begin{equation*}\label{pnt}
\lim_{t\to\infty}\frac1{\pi(n_t)}\sum_{p<n_t}F(S^px)\text{ does not exist,}
\end{equation*}
where $F(y)=(-1)^{y(0)}$.
At stage $t$, $x$ is approximated by the infinite concatenation of $x_t[0,n_t-1]\in\{0,1,?\}^{n_t}$ (that is, we see a periodic sequence of $0,1,?$ with period $n_t$). Successive ``?'' will be filled in the next steps of construction of $x$. We require that:
\begin{gather}
\label{t2}
\frac{\varphi(n_t)}{n_t}\leq\frac1{2^t},\\
\label{t3}
\{0\leq i<n_t:\: x_t(i)=?\}\subset \{0\leq j<n_t:(j,n_t)=1\},\\
\label{t4}
\#\{0\leq i<n_t:\:x_t(i)=?\}\geq\Big(1-\sum_{l=1}^t\tfrac{1}{100^l}\Big)\varphi(n_t),\\
\label{t5}
\#\{p<n_t:\:x_t(p)=?\}\geq\frac12 \pi(n_t).
\end{gather}
We choose $k_{t+1}$ satisfying~\eqref{t1} and:
\begin{gather}
\label{t6}
\frac{\varphi(k_{t+1})}{k_{t+1}}\leq\frac12,~\footnotemark\\
\label{t6 1/2}
\varphi(k_{t+1})\geq 100^{t+1},\\
\label{t6++}
8\log n_{t+1}\leq \pi(n_{t+1}),\quad 8\pi(n_{t})\leq \pi(n_{t+1})
\end{gather}
\footnotetext{Note that if $p_i$ stand for the $i$-th prime then $\sum_{j\geq i}1/p_j=+\infty$, whence remembering that $\varphi(p_ip_{i+1}\ldots p_{i+s})=p_ip_{i+1}\ldots p_{i+s}\prod_{j=0}^s(1-1/p_{i+j})$, we have $\prod_{j=0}^s(1-1/p_{i+j})\to 0$, and therefore $\prod_{j=0}^s(1-1/p_{i+j})<1/2$ for $s$ large enough.}
and {\bf for each} $0<a<n_t$, $(a,n_t)=1$, we have
\begin{equation}\label{t7}
\#\big(\{a+jn_t:\:j=0,\ldots,k_{t+1}\}\cap \PP\big)= \pi(n_{t+1};n_t,a)\leq 2\frac{\pi(n_{t+1})}{\varphi(n_t)}.
\end{equation}
The latter we obtain from  Theorem~\ref{thm:Diri} (remembering that $n_t$ is fixed, so the number of $a$ is known, we can obtain the accuracy as good as we want by taking $k_{t+1}$ sufficiently large).

We need two simple observations:

\begin{equation}\label{l1}
\{0\leq k<n_{t+1}:\:(k,n_{t+1})=1\}\subset\bigcup_{\substack{0\leq a<n_t\\ (a,n_t)=1}}
\{a+jn_t:\:j=0,\ldots,k_{t+1}-1\}.
\end{equation}

%
%

\begin{Lemma}\label{l2}
For every $0\leq a<n_t$ with $(a,n_t)=1$, we have
\[\#\{0\leq j<k_{t+1}:(a+jn_t,n_{t+1})=1\}=\varphi(k_{t+1}).\]
\end{Lemma}
\begin{proof}
First note that $(a+jn_t,n_{t+1})=1$ iff $(a+jn_t,k_{t+1})=1$. Indeed, assume that $(a+jn_t,k_{t+1})=1$.
If for some prime $p$ we have $p|(a+jn_t)$ and $p|n_{t+1}=n_tk_{t+1}$, then $p|k_{t+1}$. Otherwise, we have
$p|n_t$, so $p|a$. As $(a,n_t)=1$, this gives a contradiction. Thus $(a+jn_t,k_{t+1})=1$ implies $(a+jn_t,n_{t+1})=1$.
The opposite implication is obvious. Thus
\[\{0\leq j<k_{t+1}:(a+jn_t,n_{t+1})=1\}=\{0\leq j<k_{t+1}:(a+jn_t,k_{t+1})=1\}.\]
Let us consider the affine map
\[\Z/k_{t+1}\Z\ni j\stackrel{A}{\mapsto} a+jn_t\in \Z/k_{t+1}\Z.\]
If $J:=\{0\leq \ell<k_{t+1}:\:(\ell,k_{t+1})=1\}$ then
$$
\{0\leq j<k_{t+1}:\: (a+jn_t,k_{t+1})=1\}=A^{-1}(J).$$
Since $(n_t,k_{t+1})=1$, the map $A$ is a bijection. It follows that
\begin{align*}
\#\{0\leq j<k_{t+1}:(a+jn_t,k_{t+1})=1\}&=\#\{0\leq \ell<k_{t+1}:(\ell,k_{t+1})=1\}\\
&=\varphi(k_{t+1}),
\end{align*}
which completes the proof.
%
\end{proof}

We need to describe now which and how we fill "?" in
$x_{t+1}[0,n_{t+1}-1]$. This block is divided into $k_{t+1}$ subblocks
$$
\underbrace{x_t[0,n_t-1]x_t[0,n_t-1]\ldots x_t[0,n_t-1]}_{k_{t+1}}.$$
We fill in {\bf all} "?" in the first block $x_t[0,n_t-1]$ in such a way to ``destroy'' a PNT for the time $n_t$, namely
\begin{align*}
\frac1{\pi(n_t)}&\sum_{p<n_t}F(S^px)=\frac1{\pi(n_t)}\sum_{\substack{p<n_t\\p|n_t}}(-1)^{x(p)}+\\
&\frac1{\pi(n_t)}\Big(\sum_{\substack{p<n_t\\(p,n_t)=1\\x_t(p)=0}}1-
\sum_{\substack{p<n_t\\ (p,n_t)=1\\x_t(p)=1}}1+\sum_{\substack{p<n_t\\ (p,n_t)=1\\x_t(p)=?}}(-1)^{x(p)}\Big).
\end{align*}
As the number of the primes
dividing $n_t$ is bounded by $\log n_t$, it is negligible compared to $\pi(n_t)=n_t/\log n_t$.
It follows that
\[\Big|\frac1{\pi(n_t)}\sum_{\substack{p<n_t\\p|n_t}}(-1)^{x(p)}\Big|
\leq \frac{\log n_t}{\pi(n_t)}=o(1),\]
so the first summand does not affect the asymptotic of the averages in \eqref{pnt31}.
Since the number of $p$ in the last summand is at least $\frac12\pi(n_t)$ in view of~\eqref{t5}, we can fill in $x_{t+1}$ at places $\{p<n_t: (p,n_t)=1,\ x_t(p)=?\}$ to obtain the sum completely different that the known number which we had from stage $t-1$.

We fill in (in an arbitrary way) all remaining places between $0$ and $n_t-1$ and all places $a+jn_t$ for $0\leq j<k_{t+1}$ such that this number is not coprime with $n_{t+1}$, so that \eqref{t3} will be satisfied at stage $t+1$.
We must remember that for certain $0<a<n_t$ coprime to $n_t$, $x_t(a)$ was already defined at previous stages, so along the corresponding arithmetic progressions $a+jn_t$, $0\leq j<k_{t+1}$, these places are also filled in previously. On the other {hand},
if $x_{t+1}(a+jn_t)\neq ?$ (that is, $x_{t+1}(a+jn_t)=0$ or $x_{t+1}(a+jn_t)=1$)  and $(a+jn_t,n_{t+1})=1$ for some $0< j<k_{t+1}$ then $x_t(a)\neq ?$. In view of \eqref{l1}, it gives
\begin{align*}
\{&0\leq i<n_{t+1}:(i,n_{t+1})=1,x_{t+1}(i)\neq ?\}\\
&\subset\{0<a<n_{t}:(a,n_{t})=1,x_{t+1}(a)\neq ?\}\\
 &\quad\cup \bigcup_{\substack{0\leq a<n_t\\
(a,n_t)=1\\
x_t(a)\neq ?}}
\{a+jn_{t}:0<j<k_{t+1},\: (a+jn_t,n_{t+1})=1\}.
\end{align*}
By \eqref{t3}, Lemma~\ref{l2},  \eqref{t4} and \eqref{t6 1/2}, it follows that
\begin{align*}
\#&\{0\leq i<n_{t+1}:(i,n_{t+1})=1,x_{t+1}(i)\neq ?\}\\
&\leq\varphi(n_t)+
\#\{0\leq a<n_t:(a,n_t)=1,x_t(a)\neq ?\}\varphi(k_{t+1})\\
& \leq \varphi(n_t)+
\Big(\sum_{k=1}^t\frac{1}{100^k}\Big)\varphi(n_t)\varphi(k_{t+1})=\Big(\frac1{\varphi(k_{t+1})}+\sum_{k=1}^t\frac{1}{100^k}\Big)\varphi(n_{t+1})\\
&\leq \sum_{k=1}^{t+1}\frac{1}{100^k}\varphi(n_{t+1})\leq \frac{1}{99}\varphi(n_{t+1}).
\end{align*}
In particular, at stage $t+1$, also \eqref{t4} is satisfied.

Similar arguments show that
\begin{align*}
\{& p<n_{t+1}:x_{t+1}(p)\neq ?\}\subset \{p<n_{t+1}:p|n_{t+1}\}\cup\{p<n_{t}:x_{t+1}(p)\neq ?\}\\
&\cup \bigcup_{\substack{0\leq a<n_t\\
(a,n_t)=1\\
x_t(a)\neq ?}}
\{a+jn_{t}:0<j<k_{t+1},\: a+jn_t\in\PP\}.
\end{align*}
In view of \eqref{t7}, \eqref{t4} and \eqref{t6++}, it follows that
\begin{align*}
\#\{& p<n_{t+1}:x_{t+1}(p)\neq ?\}\\&\leq \log n_{t+1}+\pi(n_{t})
+2\#\{0\leq a<n_t: (a,n_t)=1, x_t(a)\neq ?\} \frac{\pi(n_{t+1})}{\varphi(n_t)}\\
&\leq\log n_{t+1}+\pi(n_{t})+\frac{2}{99}\varphi(n_t) \frac{\pi(n_{t+1})}{\varphi(n_t)}\leq \frac{1}{2}\pi(n_{t+1}).
\end{align*}
 Therefore, at stage $t+1$, also \eqref{t5} is satisfied.

Finally, note that
\[
\frac{\varphi(n_{t+1})}{n_{t+1}}=\frac{\varphi(n_{t})}{n_{t}}
\frac{\varphi(k_{t+1})}{k_{t+1}}\leq \frac{\varphi(n_{t})}{n_{t}}\frac12,\]
so \eqref{t2} holds and the resulting Toeplitz sequence is regular.


{\section{Toeplitz subshifts for which a P$_l$NT does not hold}\label{s:SecPlNT}

We now intend to give an  example of a (regular) Toeplitz sequence $x$ such that a P$_l$NT does not hold for {the corresponding subshift. In fact,}
\begin{equation*}\label{plnt}
\lim_{t\to\infty}\frac1{\pi_l(n_t)}\sum_{p^{(l)}\in\mathbb{P}_l^{n_t}}F(S^{p^{(l)}}x)\text{ does not exist.}
\end{equation*}
For any natural $m$ and $0\leq a<m$, let
\[\pi_l(N;m,a):=\#(\mathbb{P}_l^{N}\cap(a+m\Z)).\]

\begin{Lemma}\label{lem:al1}
If $(a,m)>1$ then
\begin{equation}\label{eq:am>1}
\pi_l(N;m,a)=o(\pi_l(N)).
\end{equation}
If $(a,m)=1$ then
\begin{equation}\label{eq:am1}
\lim_{N\to\infty}\frac{\pi_l(N;m,a)}{\pi_l(N)}=\frac{1}{\varphi(m)}.
\end{equation}
\end{Lemma}

\begin{proof}
The proof is by induction on $l$. If $l=1$ and $(a,m)>1$ then
$\pi_l(N;m,a)\leq 1$, so \eqref{eq:am>1} holds. If  $(a,m)=1$ then \eqref{eq:am1} is given by Theorem~\ref{thm:Diri}.

Suppose that \eqref{eq:am>1} and \eqref{eq:am1} are satisfied for all parameters less than some natural number $l\geq 2$.
Assume that $(a,m)\in \mathbb{P}_j$ for some $j\geq 1$. If $j>l$ then $\pi_l(N;m,a)=0$.
If $(a,m)\in \mathbb{P}_l$ then $\pi_l(N;m,a)\leq 1$, so \eqref{eq:am>1} holds.
If $p^{(j)}:=(a,m)\in \mathbb{P}_j$ for some $1\leq j< l$ then,
\begin{align*}
\pi_l(N;m,a)&\leq \pi_{l-j}([N/p^{(j)}];m/p^{(j)},a/p^{(j)})=O(\pi_{l-j}(N))\\
&=O\Big(\frac{N(\log\log N)^{l-j-1}}{\log N}\Big)=o\Big(\frac{N(\log\log N)^{l-1}}{\log N}\Big)=o(\pi_l(N)).
\end{align*}

Now, suppose that $(a,m)=1$. Assume that $p_1\leq p_2\leq\ldots\leq p_l$ are prime numbers such that $p^{(l)}=p_1\cdots p_l\leq N$, $p^{(l)}=a\!\mod m$.
Then $p_1\leq \sqrt[l]{N}$. Since $(p_1,m)=1$, there exists a unique $0\leq a(p_1)<m$ such that $p_1\cdot a(p_1)=a\!\mod m$ and $(a(p_1),m)=1$.
Then
\[\pi_l(N;m,a)=\sum_{p_1\leq \sqrt[l]{N}}\pi_{l-1}([N/p_1];m,a(p_1)).\]
As $p_1\leq \sqrt[l]{N}$ implies $N/p_1\geq {N}^{1-1/l}$, by assumption, for every $\vep>0$ there exists $N_\vep$
such that for all $N\geq N_\vep$ and $p_1\leq \sqrt[l]{N}$ with $(p_1,m)=1$, we have
\[(1-\vep)\frac{\pi_{l-1}([N/p_1])}{\varphi(m)}<\pi_{l-1}([N/p_1];m,a(p_1))<(1+\vep)\frac{\pi_{l-1}([N/p_1])}{\varphi(m)}.\]
Since $\pi_l(N)=\sum_{p_1\leq \sqrt[l]{N}}\pi_{l-1}([N/p_1])$, it follows that
\[(1-\vep)\frac{\pi_l(N)}{\varphi(m)}<\pi_l(N;m,a)<(1+\vep)\frac{\pi_l(N)}{\varphi(m)}\]
for every $N\geq N_\vep$, so we have \eqref{eq:am1}.
\end{proof}

\begin{Lemma}\label{lem:al2}
For every $l\geq 2$, we have
\begin{equation}
\#\{p^{(l)}\in \mathbb{P}_l^{N}:(p^{(l)},N)>1\}=o(\pi_l(N)).
\end{equation}
\end{Lemma}
\begin{proof}
{Notice that $\#\{p^{(l)}\in \mathbb{P}_l^{N}\!:\!(p^{(l)},N)>1\}\leq \sum_{p|N} \pi_{l-1}(\frac{N}{p})$. Therefore,  using \eqref{land},
\begin{align*}
\#\{p^{(l)}\in \mathbb{P}_l^{N}:(p^{(l)},N)>1\}&={O}\Big(\sum_{p|N} \frac{N/p}{\log (N/p)}\frac{(\log\log (N/p))^{l-2}}{(l-2)!}\Big)\\
&=
{O}\Big(\frac{N}{\log N} \frac{(\log\log (N))^{l-1}}{(l-1)!} \frac{(l-1)}{\log \log N}\sum_{p|N} \frac{\log N}{p\log (N/p)}\Big).
\end{align*}
So again, by \eqref{land}, the result will follow by showing that
\[
\frac{1}{\log \log N}\sum_{p|N} \frac{\log N}{p\log (N/p)}=o(1).
\]
Note that
\[
\sum_{p|N} \frac{\log N}{p\log (N/p)}=\sum_{\substack{p|N\\p\leq N^{1/2}}} \frac{\log N}{p\log (N/p)}+\sum_{\substack{p|N\\p> N^{1/2}}} \frac{\log N}{p\log (N/p)},
\]
and that the second term contains at most one prime $p$. Moreover, as $l\geq 2$ and $(p^{(l)}, N)>1$, the number $N$ is not prime, so $N/p\geq 2$.   Using this, we get
\[
\sum_{p|N} \frac{\log N}{p\log (N/p)}\leq {2}\sum_{p|N}\frac{1}{p}+\frac{1}{\log 2}\frac{\log N}{N^{1/2}}={O}(\log \log \log N),
\]
as $\sum_{p|N}\frac{1}{p}={O}(\log \log \log N)$, see e.g.\ \cite{KLR0}. This finishes the proof.}
\end{proof}

Now, we repeat the scheme of the construction from Section~\ref{s:Sec1} almost word for word, although we have to take care how to choose $k_{t+1}$.

First of all, we require that $k_{t+1}$ is
large enough so that
\begin{gather}
\pi_l(n_{t+1};n_t,a)\leq 2\frac{\pi_l(n_{t+1})}{\varphi(n_t)}\text{ for every }0\leq a<n_t\text{ with } (a,n_t)=1,\label{eq:g1}\\
\sum_{0\leq a<n_t,(a,n_t)>1}\pi_l(n_{t+1};n_t,a)\leq \frac{\vep}{8}\pi_l(n_{t+1}),\label{eq:g2}\\
\#\{p^{(l)}\in \mathbb{P}_l^{n_{t+1}},(p^{(l)},n_{t+1})>1\}=o(\pi_l(n_{t+1})).\label{eq:g3}
\end{gather}
The existence of such $k_{t+1}$ is guaranteed by Lemmas~\ref{lem:al1}~and~\ref{lem:al2}.

Next, we replace~\eqref{t5} by
\begin{equation*}\label{t5a}
\#\{p^{(l)}\in\PP_l^{n_t}:\:x_t(p^{(l)})=?\}\geq\frac12\pi_l(n_t)
\end{equation*}
and requiring (instead of~\eqref{t7}) that for
$(a,n_t)=1$, we have
\begin{equation*}\label{t7a}
\#(\{a+jn_t:\:0\leq j <k_{t+1}\}\cap \PP_l)\leq 2 \frac{\pi_l(n_{t+1})}{\varphi(n_t)},
\end{equation*}
cf.\ \eqref{eq:g1}.
Furthermore, we replace~\eqref{t6++} by the requirement that
\[
\#\{p^{(l)}\in\PP_l^{n_t}:p^{(l)} \equiv a \operatorname{mod}{n_t}\text{ with }(a,n_t)>1\}\leq\frac18\pi_l(n_{t+1}),\]
cf.\ \eqref{eq:g2}.
To carry over the previous proof, it remains to show that
\[
\frac1{\pi_{l}(n_t)}\sum_{p^{(l)}\in\PP_l^{n_t},(p^{(l)},n_t)>1}(-1)^{x(p^{(l)})}
=o(1).\]
This follows from \eqref{eq:g3} applied in the previous step of the construction.}

\section{Regular Toeplitz subshifts which satisfy a PNT (proof of Theorem~B)}\label{s:Sec3}
Let $x\in\mathcal{A}^\Z$ be a regular Toeplitz sequence. Then, for every $k\in\N$, there is
an $n_k$-periodic sequence $x_k\in(\mathcal{A}\cup\{?\})^\Z$ so that
\[x_k(j)\neq ?\text{ implies }x(j)=x_k(j)=x_l(j)\text{ for all }l\geq k\]
and
\[?_k=?_k(x):=\#\{0\leq j<n_k:x_k(j)=?\}=o(n_k).\]

For every Toeplitz sequence $x\in\mathcal{A}^\Z$ and natural $m$ let us consider a new Toeplitz
sequence $x^{(m)}\in(\mathcal{A}^{2m+1})^\Z$ given by
\[x^{(m)}(j)=(x(j-m),\ldots,x(j+m))\text{ for every }j\in\Z.\]
If $(n_t)_{t\geq 1}$ is a periodic structure of $x$, then it is also a periodic structure of $x^{(m)}$. Moreover,
\begin{equation}\label{ineq:?}
?_k(x^{(m)})\leq (2m+1)?_k(x)\text{ for every }k\geq 1.
\end{equation}
Hence, the regularity of $x$ implies the regularity of $x^{(m)}$.


Theorem~B follows directly from Lemma~\ref{lem:?Et} and the following result.

\begin{Th}\label{thm:PNT}
Suppose that $(X_x,S)$ is a Toeplitz system such that
\begin{equation*}
?_k=o(\varphi(n_k)).
\end{equation*}
Then $(X_x,S)$ satisfies a PNT.
\end{Th}
\begin{proof}
To show a PNT for $(X_x,S)$, {it suffices to show that} for every continuous $F:X_x\to\C$ and every
$\vep>0$ there exists $N_{\vep}$ so that for every $N,M\geq N_{\vep}$ and every $r\in\Z$, we have
\begin{equation}\label{eq:cauch}
\Big|\frac1{\pi(N)}\sum_{p\leq N}F(S^{p+r}x)- \frac1{\pi(M)}\sum_{p\leq M}F(S^{p+r}x)\Big|<\vep.
\end{equation}
{Note that the above is stronger than what is needed as it shows that the convergence in~\eqref{pnt31} is uniform in $y\in X_x$.}
We  first assume that $F:X_x\to\R$ depends only on the zero coordinate, i.e.\ $F(y)=f(y(0))$ for some $f:\mathcal{A}\to\R$.

Fix $\vep>0$. Fix also $k\geq 1$ so that
\begin{equation}\label{tralala}
?_k<\frac{\vep}{8}\varphi(n_k).
\end{equation}
Next choose $N_\vep$ such that for every $N\geq N_\vep$, we have
\begin{gather}
\label{eq:dir}
\Big|\pi(N;n_k,a)-\frac{\pi(N)}{\varphi(n_k)}\Big|<\frac{\vep}{8}\frac{\pi(N)}{\varphi(n_k)}\text{ for all }a\in\Z\text{ with }(a,n_k)=1,\\
\label{eq:loglog}
\#\{p\leq N:p| n_{k}\}\leq \log n_k<\frac{\vep}{8}\pi(N).
\end{gather}
We will show that for all $N\geq N_\vep$ and $r\in\Z$ we have
\begin{equation}\label{eq:p-a}
\Big|\frac{1}{\pi(N)}\sum_{p\leq N}F(S^{p+r}x)-\frac{1}{\varphi(n_{k})}\sum_{\substack{0\leq a<n_{k}\\(a-r,n_{k})=1\\x_k(a)\neq ?}}F(S^{a}x)\Big|\leq \vep\|F\|_{\sup},
\end{equation}
which implies \eqref{eq:cauch}.

Recall that $x_{k}\in(\mathcal{A}\cup\{?\})^\Z$ is an $n_{k}$-periodic sequence (used to construct $x$ at stage $k$). If for some $a\in\Z$ we have
\[
x_{k}(a)\neq\; ?,
\]
then
\[
x(a+j\cdot n_{k})=x_{k}(a)\text{ for every }j\in\Z.
\]
This implies that if $p\leq N$ and $x_{k}(p+r \!\mod n_{k})\neq \;?$, then
\begin{equation}\label{eq:z2}
F(S^{p+r}x)=F(S^{p+r \!\mod n_{k}}x).
\end{equation}
Note that
\begin{align*}
\#\{&p\leq N: x_{k}(p+r\!\mod n_{k})=\;?\}\\
&\leq
 \sum_{\substack{0\leq a<n_k\\
(a-r,n_k)=1\\
x_k(a)=?}}\#\{p\leq N: p=a-r\!\mod n_{k}\}\\
&\quad+
 \sum_{\substack{0\leq a<n_k\\
(a-r,n_k)>1}}\#\{p\leq N: p=a-r\!\mod n_{k}\}.
\end{align*}

Assume that $N\geq N_\vep$. By \eqref{eq:dir} and \eqref{eq:loglog}, for every integer $v$ with $(v,n_k)=1$ we have
\[\#\{p\leq N:p=v\!\mod n_{k}\}=\pi(N;n_k,v)\leq (1+\vep/8)\frac{\pi(N)}{\varphi(n_k)}\]
and
\begin{equation}\label{eq:()>1}
\sum_{\substack{0\leq a<n_k\\
(a-r,n_k)>1}}\#\{p\leq N: p=a-r\!\mod n_{k}\}\leq \#\{p\leq N:p| n_{k}\}<\frac{\vep}{8}\pi(N),
\end{equation}
where left inequality follows from the fact that if $(a-r,n_k)>1$ and $p_a=a-r \!\mod n_{k}$ for a prime $p_a$, then $(a-r,n_k)=p_a$ and
\[\{p\leq N:\:p=a-r \!\mod n_{k}\}=\{p_a\}.\]
It follows that (use also~\eqref{tralala})
\begin{align*}
\#&\{p\leq N: x_{k}(p+r\!\mod n_{k})=\;?\}\\
&\leq \#\{0\leq a<n_k:(a-r,n_k)=1,x_k(a)=?\}(1+\vep/8)\frac{\pi(N)}{\varphi(n_k)}+\frac{\vep}{8}\pi(N)
\\
&\leq ?_k(1+\vep/8)\frac{\pi(N)}{\varphi(n_k)}+\frac{\vep}{8}\pi(N)\leq \frac{\vep}{2}\pi(N).
\end{align*}
Let
\[
P_N:=\{p\leq N: x_{k}(p+r\!\mod n_{k})\neq\;?\}.
\]
Then by the above, for every $N\geq N_{\vep}$,
\begin{align}\label{eq:P_N}
\Big|\frac{1}{\pi(N)}\sum_{p\leq N}F(S^{p+r}x)- \frac{1}{\pi(N)}\sum_{p\in P_N}F(S^{p+r}x)\Big|\leq \frac{\vep}{2}\|F\|_{\sup}.
\end{align}
But by \eqref{eq:z2},
\begin{align*}
\sum_{p\in P_N}F(S^{p+r}x)&=\sum_{\substack{0\leq a<n_{k}\\
x_k(a)\neq ?}}\sum_{\substack{p\leq N\\p\equiv a-r \!\mod n_{k}}}F(S^{a}x)\\
&=
\sum_{\substack{0\leq a<n_{k}\\x_k(a)\neq ?}}F(S^{a}x)\#\{p\leq N,p=a-r \!\mod n_{k}\}.
\end{align*}
If $(a-r,n_k)=1$, then again by \eqref{eq:dir}, we have
\[\Big|\#\{p\leq N,p=a-r \!\mod n_{k}\}-\frac{\pi(N)}{\varphi(n_k)}\Big|=\Big|\pi(N;n_k,a-r)-
\frac{\pi(N)}{\varphi(n_k)}\Big|<
\frac{\vep}{8}\frac{\pi(N)}{\varphi(n_k)}.\]
In view of \eqref{eq:()>1}, it follows that
\begin{align*}
&\Big|\frac{1}{\pi(N)}\sum_{p\in P_N}F(S^{p+r}x)-\frac{1}{\varphi(n_k)}\sum_{\substack{0\leq a<n_{k}\\(a-r,n_k)=1\\x_k(a)\neq ?}}F(S^{a}x)\Big|\\
&=\Big|\sum_{\substack{0\leq a<n_{k}\\x_k(a)\neq ?}}F(S^{a}x)\frac{\pi(N;n_k,a-r)}{\pi(N)}-\frac{1}{\varphi(n_k)}\sum_{\substack{0\leq a<n_{k}\\
(a-r,n_k)=1\\x_k(a)\neq ?}}F(S^{a}x)\Big|\\
&\leq\frac{1}{\pi(N)}\sum_{\substack{0\leq a<n_{k}\\(a-r,n_k)=1\\x_k(a)\neq ?}}|F(S^{a}x)|\Big|\pi(N;n_k,a-r)-\frac{\pi(N)}{\varphi(n_k)}\Big|+\frac{\vep}{8}\|F\|_{\sup}\\
&\leq\|F\|_{\sup}\Big(\frac{\vep}{8}\frac{\#\{0\leq a<n_{k}:x_k(a)\neq ?,(a-r,n_k)=1\}}{\varphi(n_k)}+\frac{\vep}{8}\Big)
\leq \|F\|_{\sup}\frac{\vep}{2}.
\end{align*}
Together with \eqref{eq:P_N}, this gives \eqref{eq:p-a}, which completes the proof in the case of $F$ depending only on the zero coordinate.

Now suppose that $F:X_x\to\C$ depends only on finitely many coordinates. Then there exists natural $m$ and $f:\mathcal{A}^{2m+1}\to\C$ such that
$F(y)=f(y(-m),\ldots,y(m))$ for every $y=(y(k))_{k\in\Z}\in X_x$.
Denote by $X_{x^{(m)}}\subset(\mathcal{A}^{2m+1})^{\Z}$ the orbit closure
of $x^{(m)}\in (\mathcal{A}^{2m+1})^{\Z}$. Then every $y^{(m)}\in X_{x^{(m)}}$ is of the form $y^{(m)}(k)=(y(k-m),\ldots,y(k+m))$ for some $y=(y(k))_{k\in\Z}\in X_x$.

In view of \eqref{ineq:?}, $(X_{x^{(m)}},S)$ is a regular Toeplitz shift with $?_k(x^{(m)})=o(\varphi(n_k))$.
Let us consider $\bar{F}:X_{x^{(m)}}\to\C$ given by $\bar{F}(y^{(m)})=f(y^{(m)}(0))=f(y(-m),\ldots,y(m))$ for $y^{(m)}\in X_{x^{(m)}}$.
Since $\bar{F}$ depends only on the zero coordinate, by \eqref{eq:cauch} applied to $x^{(m)}$ and the map $\bar{F}$, for every $\vep>0$ there exists $N_\vep$ such that for $N,M\geq N_\vep$, we have
\begin{align*}
\Big|&\frac1{\pi(N)}\sum_{p\leq N}F(S^{p+r}x)- \frac1{\pi(M)}\sum_{p\leq M}F(S^{p+r}x)\Big|\\
&=
\Big|\frac1{\pi(N)}\sum_{p\leq N}\bar{F}(S^{p+r}x^{(m)})- \frac1{\pi(M)}\sum_{p\leq M}\bar{F}(S^{p+r}x^{(m)})\Big|<\vep.
\end{align*}
Thus \eqref{eq:cauch} holds for every $F:X_x\to\C$ depending only on finitely many coordinates. As the set of such functions is dense in $C(X_x)$,  \eqref{eq:cauch} also holds for every $F\in C(X_x)$, which completes the proof.
\end{proof}

As $\varphi(n)\to\infty$ when $n\to\infty$, we obtain the following result.

\begin{Cor}\label{c:finite?}
If $x$ is Toeplitz  for which the sequence $(?_k)$ is bounded then $(X_x,S)$ satisfies a PNT.\end{Cor}

{
\section{Toeplitz subshifts for which a P$_l$NT holds}\label{s:Sec4}
\begin{Th}\label{thm:PlNT}
Suppose that $(X_x,S)$ is a Toeplitz system such that
\begin{equation*}
?_k=o(\varphi(n_k)).
\end{equation*}
Then, for every $F\in C(X_x)$ and $y\in X_x$, the limit
\[\lim_{N\to\infty}\frac{1}{\pi_l(N)}\sum_{p^{(l)}\in\PP_l^{N}}F(S^{p^{(l)}}y)\text{ exists.}\]
\end{Th}

\begin{proof}
The proof proceeds along the same lines as the proof of Theorem~\ref{thm:PNT}. It relies on the
following analogue of \eqref{eq:p-a}: for every $\vep>0$ there exists a natural $N_\vep$ such that for all $N\geq N_\vep$ and $r\in\Z$, we have
\begin{equation}\label{eq:pp-a}
\Big|\frac{1}{\pi_l(N)}\sum_{p^{(l)}\in\PP_l^{N}}F(S^{p^{(l)}+r}x)-\frac{1}{\varphi(n_{k})}\sum_{\substack{0\leq a<n_{k}\\(a-r,n_{k})=1\\x_k(a)\neq ?}}F(S^{a}x)\Big|\leq \vep\|F\|_{\sup}.
\end{equation}
In turn, the proof of \eqref{eq:p-a} is based on only two elements: \eqref{eq:dir} and \eqref{eq:()>1}. Their {$l$-almost prime}
 counterparts follow directly from \eqref{eq:am1} and \eqref{eq:am>1}, respectively. Now, we repeat
the arguments of the proof of \eqref{eq:p-a} almost word for word, replacing \eqref{eq:dir} and \eqref{eq:()>1} by their
$l$-almost prime counterparts.
\end{proof}

\begin{Remark}
In view of \eqref{eq:p-a} and \eqref{eq:pp-a}, under the assumption $?_k=o(\varphi(n_k))$, we have
\[\lim_{N\to\infty}\frac{1}{\pi_l(N)}\sum_{p^{(l)}\in\PP_l^N}F(S^{p^{(l)}}y)=\lim_{N\to\infty}\frac{1}{\pi(N)}\sum_{p<N}F(S^{p}y)\]
for every $F\in C(X_x)$ and $y\in X_x$, so a PNT and a P$_l$NT fully coincide for this class of regular Toeplitz systems.
\end{Remark}}

\section{Ergodic averages along polynomial times}\label{s:Sec5}
Let $P$ be a monic polynomial\footnote{The leading coefficient of $P$ equals~1. This assumption is only for simplicity. In fact, Theorem~\ref{thm:posP} below is true whenever the set of (non-zero) coefficients of $P-P(0)$ is coprime, see the proof of Corollary~\ref{cor:rhoP} and the assumptions of Albis theorem in \cite{Nar}.} of degree $d>1$ with non-negative integer coefficients. Note that, under these assumptions, $P(\cdot)$ is a strictly increasing function on $\N$.
For every $n\in\N$, let
\[R^{P}_n:=\{0\leq a<n: a=P(m)\!\mod n\text{ for some }m\in\N\}\text{ and }\psi^P(n):=\# R^P_n.\]
For all $N,n\in\N$ and $a\in R^P_n$, let
\[\rho^P(N;n,a)=\#\{1\leq m\leq N:P(m)=a\!\mod n\}.\]
and
\[\rho^P(n,a):=\rho^P(n;n,a),\;\;\rho^P(n):=\max_{a\in R^P_n}\rho^P(n;n,a).\]

\begin{Lemma}\label{lem:psi}
The function $\psi^P$ is multiplicative, i.e.\ $\psi^P(n_1n_2)=\psi^P(n_1)\psi^P(n_2)$ if $(n_1,n_2)=1$.
If $a\in\Z/n\Z$, $n_1,\ldots,n_k$ are pairwise coprime and $n=n_1\cdots n_k$ then $a\in R^P_{n}$ iff $a_i\in R^P_{n_i}$ for $i=1,\ldots,k$,
where $0\leq a_i<n_i$ is the remainder of $a$ when divided by $n_i$ (that is, $0\leq a_i<n_i$ and $a_i=a$ mod~$n_i$). Moreover,
\begin{equation}\label{eq:rhoP}
\rho^P(n,a)=\prod_{i=1}^k\rho^P(n_i,a_i).
\end{equation}
\end{Lemma}

\begin{proof}
Note that the multiplicativity of $\psi^P$ follows from the second part of the lemma.

Moreover, note that $a\in R^P_n$ iff  $a=P(m)\!\mod n$ for some $0\leq m<n$. Indeed, if $a=P(m)\!\mod n$ for some $m\in\N$, then  $a=P(m')\!\mod n$, where $0\leq m'<n$
is the remainder of $m$ when divided by $n$.

If $a\in R^P_{n_1\cdots n_k}$, i.e.\ $a=P(m)\!\mod n_1\cdots n_k$ for some $0\leq m<n$, then $a_i=a=P(m)=P(m_i)\!\mod n_i$ for every $i=1,\ldots,k$,  where $0\leq m_i<n_i$
is the remainder of $m$ when divided by $n_i$.

Now, suppose $a\in \Z/n\Z$, $a_i=a$ mod~$n_i$ and $a_i\in R^P_{n_i}$ for $i=1,\ldots,k$. Then, for every  $i=1,\ldots,k$, there exists $0\leq m_i<n_i$ such that $a_i=P(m_i)\!\mod n_i$.
By the Chinese Remainder Theorem, there exists a unique $0\leq m<n$ such that $m=m_i\!\mod n_i$ for $i=1,\ldots,k$. It follows that
\[P(m)=P(m_i)=a_i=a\!\mod n_i\text{ for all }i=1,\ldots,k.\]
This yields $a=P(m)\!\mod n_1\cdots n_k$ and $a\in R^P_n$.

The argument above also shows \eqref{eq:rhoP}.
\end{proof}

\begin{Remark}\label{r:psi} Note that in the argument above we used the fact that the $a_i$'s determine $a$ as by the ChRT there exists only one $0\leq a<n$ such that $a=a_i$ mod~$n_i$ for each $i=1,\ldots,k$.\end{Remark}

For any natural $n$ denote by $\omega(n)$ the number of its prime divisors (counted without multiplicities) and by $p(n)$ the product of its prime divisors.
\begin{Cor}\label{cor:rhoP}
The arithmetic function $\rho_P$ is multiplicative and $\rho^P(n)\leq \frac{d^{\omega(n)}}{p(n)}n$.
\end{Cor}
\begin{proof}
The multiplicativity of $\rho^P$ follows directly from \eqref{eq:rhoP}. By Albis theorem (see Corollary 3 of Theorem~1.23 in \cite{Nar}~\footnote{Note that compared to notation from \cite{Nar}, we have:
$$
\rho^P(n;n,a)=\lambda_{P-a}(n),\;\rho^P(n)=\max_{a\in R^P_n}\lambda_{P-a}(n);$$
the estimate on $\lambda_{P}$ in \cite{Nar} depends {\bf only} on the degree of the polynomial.}),
for any prime number we have $\rho^P(p^n)\leq d p^{n-1}$. This result combined with the multiplicativity of $\rho^P$ gives the required bound of
$\rho^P(n)$.
\end{proof}

\begin{Lemma}\label{lem:Nna}
For all $n\in\N$, $a\in R^P_n$ and $N\geq P(n)$, we have
\begin{align*}\rho^P(n,a)\Big(\frac{P^{-1}(N)}{n}-1\Big)&
\leq\#\{m\in\N:1\leq P(m)\leq N,P(m)=a\!\mod n\}\\
&\leq \rho^P(n,a)\Big(\frac{P^{-1}(N)}{n}+1\Big).
\end{align*}
\end{Lemma}

\begin{proof}
Let $s:=\rho^P(n,a)$ and let $1\leq m_1<\ldots<m_s\leq n$ be all numbers such that $P(m_i)=a\!\mod n$.
Note that a natural number $m$ satisfies $P(m)\leq N$ and $P(m)=a\!\mod n$ iff $m=jn+r$ with $0\leq j\leq (P^{-1}(N)-r)/n$ and $0 <r\leq n$ satisfies  $P(r)=a\!\mod n$.
Thus, $r=m_i$ for some $i=1,\ldots,s$. It follows that
\begin{align*}
\rho:&=\#\{m\in\N:1\leq P(m)\leq N, P(m)=a\!\mod n\}\\
&=\sum_{i=1}^s\Big(\Big[\frac{P^{-1}(N)-m_i}{n}\Big]+1\Big).
\end{align*}
Since
\begin{align*}\frac{P^{-1}(N)}{n}-1&\leq\frac{P^{-1}(N)-m_i}{n}<\Big[\frac{P^{-1}(N)-m_i}{n}\Big]+1\\
&\leq \frac{P^{-1}(N)-m_i}{n}+1<\frac{P^{-1}(N)}{n}+1,
\end{align*}
by summing up, this gives
\[s\Big(\frac{P^{-1}(N)}{n}-1\Big)\leq\rho\leq s\Big(\frac{P^{-1}(N)}{n}+1\Big).\]
\end{proof}

\begin{Remark}\label{r:uzycie}  As $P$ is an increasing function, we can apply the above inequalities to $P(N)$ instead of $N$  (as $P(N)\geq N$). Then $P(m)\leq P(N)$ iff $m\leq N$, and the result of the lemma implies
$$
\rho^P(n,a)\Big(\frac Nn-1\Big)\leq\rho^P(N;n,a)\leq \rho^P(n,a)\Big(\frac Nn+1\Big).$$
\end{Remark}

We now focus on the simplest case when $P(n)=n^2$. We continue to write $R$ for $R^P$, $\psi$ for $\psi^P$ and $\rho$ for $\rho^P$.
In view of Theorems~1.27~and~1.30 in \cite{Nar}, we have the following result.

\begin{Prop}\label{prop:rhop}
For every prime number $p>2$,
for every $a\in R_{p^{N}}$, where $N=2n$ or $2n+1$, we have \begin{equation*}\label{eq:p2p}
\rho(p^{N},a)=\left\{\begin{array}{cl}
2 & \text{ if } a=a'\!\mod p\text{ for }a'\in R_{p}\setminus\{0\}\\
2p^r & \text{ if } a=p^{2r}a'\text{ and }a'=a''\!\mod p\text{ for }a''\in R_{p}\setminus\{0\}\\
p^n & \text{ if } a= 0.
\end{array}\right.
\end{equation*}
Moreover,  we have
\[\psi(p^{2n+1})=\frac{p^{2n+2}+2p+1}{2(p+1)}\text{ and }\psi(p^{2n})=\frac{p^{2n+1}+p+2}{2(p+1)}.~\footnote{We obtain these formulas by using the formulas for the values of $\sigma$ and counting the number of the possibilites in each row, so for $N=2n$, we have:
$$
\psi(p^{2n})=\frac{p-1}2p^{2n-1}+\sum_{r=1}^{n-1}\frac{p-1}2p^{2n-2r-1}+1
=$$$$
1+p\frac{p-1}2\sum_{r=0}^{n-1}p^{2(n-r-1)}=1+\frac{p(p-1)}2
\frac{(p^2)^n-1}{p^2-1}=\frac{p^{2n+1}+p+2}{2(p+1)}.$$
}
\]
Furthermore,
If $p=2$ then
\[
\rho(2,a)=1\text{ for all }a\in R_2,\;\;
\rho(4,a)=2\text{ for all }a\in R_4
\]
and for any $N\geq 3$, where $N=2n$ or $2n+1$, for every $a\in R_{2^{N}}$, we have
\begin{equation*}\label{eq:p2p2}
\rho(2^{N},a)=\left\{\begin{array}{cl}
4 & \text{ if } a=1\!\mod 8\\
4\cdot2^r & \text{ if } a=2^{2r}a',\,2r\leq N-3,\,a'=1\!\mod 8\\
2\cdot 2^r & \text{ if } a=2^{2r}a',\,2r= N-2,\,a'=1\!\mod 4\\
2^r & \text{ if } a=2^{2r}a',\,2r= N-1,\,a'=1\!\mod 2\\
2^n & \text{ if } a=0.
\end{array}\right.
\end{equation*}
Moreover, \[\psi(2^{2n})=\frac{2^{2n-1}+4}{3}\text{ and }\psi(2^{2n+1})=\frac{2^{2n}+5}{3}.\]
\end{Prop}

%
%

\begin{Cor}\label{cor:n^2}
For every natural $n\geq 2$, we have $\rho(n)\leq 4\sqrt{n}$. Moreover, if $n$ is square-free, then $\rho(n)\leq 2^{\omega(n)}$.
\end{Cor}
\begin{proof} By a direct inspection of the formulas in Proposition~\ref{prop:rhop}, we obtain:
$$
\rho(2^N)\leq 2\sqrt{2^N},\;\rho(3^N)\leq 2 \sqrt{3^N}$$
but for all $p\geq5$, we have
$$
\rho(p^N)\leq \sqrt{p^N}.$$
Indeed, for the cases $a=a'$ mod~$p$ (for $a'\in R_p\setminus\{0\}$) and $a=0$, it is direct. For the case $\rho(p^N,a)=2p^r$, we have $a=p^{2r}a'<p^N$, so $2r\leq N-1$ and then indeed $2p^r\leq p^{N/2}$.

The second inequality follows directly from $\rho(p)\leq 2$.
\end{proof}


%

For some future purposes, we are interested in cases (in Proposition~\ref{prop:rhop}) which gives possibly smallest values for the function $\rho$, hence,
for every prime number $p$ and any natural $n$, let
\[\widetilde{R}_{p^n}:=\left\{
\begin{array}{cl}
\{0\leq a<p^n:a=a'\!\mod p\text{ for }a'\in R_p\setminus\{0\}\}&\text{ if }p>2\\
R_2 &\text{ if }p^n=2\\
R_4 &\text{ if }p^n=4\\
\{0\leq a<2^n:a=1\!\mod 8\}&\text{ if }n\geq 3.
\end{array}
\right.\]
By Proposition~\ref{prop:rhop}, $\widetilde{R}_{p^n}\subset{R}_{p^n}$.

Let $n=p_1^{m_1}p_2^{m_2}\cdots p_k^{m_k}$ be the canonical representation of $n$.
Let
\[\Phi:\Z/n\Z\to \Z/p_1^{m_1}\Z\times\ldots\times \Z/p_k^{m_k}\Z\]
be the canonical ring isomorphism. Recall (cf.\ Lemma~\ref{lem:psi} and Remark~\ref{r:psi}) that $\Phi$ establishes a
one-to-one correspondence between $R_n$ and $R_{p_1^{m_1}}\times\ldots\times R_{p_k^{m_k}}$.
Set
\[\widetilde{R}_{n}:=\Phi^{-1}(\widetilde R_{p_1^{m_1}}\times\ldots\times \widetilde R_{p_k^{m_k}})\]
and
\[\widetilde \psi(n):=\#\widetilde{R}_{n}.\]
Then, clearly, $\widetilde \psi$ is a multiplicative function. Moreover, by  Proposition~\ref{prop:rhop}, for each $a\in \widetilde{R}_{p^N}$, we have
$$
\rho(p^N,a)=\left\{
\begin{array}{cl}
1&\text{ if }p^N=2\\
2 &\text{ if }p^N=2\text{ or }p>2\\
4 &\text{ if }p=2\text{ and }N\geq3.
\end{array}
\right.
$$
Hence, in view of~\eqref{eq:rhoP},
for every $a\in \widetilde R_n$, we have
\begin{equation}\label{eq:omegarho}
\frac12\cdot 2^{\omega(n)}\leq \rho(n,a)\leq 2\cdot 2^{\omega(n)}.
\end{equation}
Moreover, by definition,
\[\widetilde{\psi}(p^n):=\left\{
\begin{array}{cl}
p^{n-1}\frac{p-1}{2}&\text{ if }p>2\\
2 &\text{ if }p^n=2\\
2 &\text{ if }p^n=4\\
2^{n-3}&\text{ if }p=2\text{ and }n\geq 3.
\end{array}
\right.\]
It follows that
\begin{equation}\label{eq:tpsi}
\frac{1}{2}\prod_{p|n}\Big(1-\frac{1}{p}\Big)
\leq \frac{2^{\omega(n)}\widetilde{\psi}(n)}{n}\leq 4\prod_{p|n}\Big(1-\frac{1}{p}\Big).
\end{equation}
(To obtain these inequalities, for $n=p_1^{m_1}p_2^{m_2}\cdots p_k^{m_k}$, write $\frac{2^{\omega(n)}\widetilde{\psi}(n)}{n}=\prod_{i=1}^k
\frac{2\tilde{\psi}(p_i^{m_i})}{p_i^{m_i}}$ and apply the formula above.)

{\subsection{Polynomial ergodic theorem}\label{sec:nonconv}
In the result below $P$ is a monic polynomial of degree $d>1$ with non-negative integer coefficients.}
\begin{Th}\label{thm:posP}
Suppose that $(X_x,S)$ is a Toeplitz system such that
\begin{equation}\label{assum:?}
?_k=o(n_k/{\rho^P(n_k)}).
\end{equation}
Then, for every continuous map $F:X_x\to\C$ and $y\in X_x$,  the limit
\begin{equation}\label{eq:SNT}
\lim_{N\to\infty}\frac{1}{N}\sum_{m\leq N}F(S^{P(m)}y)
\end{equation}
exists.
\end{Th}
\begin{proof}
To show \eqref{eq:SNT}, we need to prove that for every
$\vep>0$ there exists $N_{\vep}$ so that for every $N,M\geq N_{\vep}$ and every $r\in\Z$, we have

\begin{equation}\label{eq:cauch^2}
\Big|\frac1{N}\sum_{m\leq N}F(S^{P(m)+r}x)- \frac1{M}\sum_{m\leq M}F(S^{P(m)+r}x)\Big|<\vep.
\end{equation}
We  first assume that $F:X_x\to\R$ depends only on the zero coordinate, i.e.\ $F(y)=f(y(0))$ for some $f:\mathcal{A}\to\R$.

Fix $\vep>0$. Choose $k\geq 1$ so that
\begin{equation}\label{eq:?}
?_k<\frac{\vep}{8}\frac{n_k}{{\rho^P(n_k)}}.
\end{equation}
Next, choose $N_\vep\geq 8n_k^2/\vep$. Then, in view of Remark~\ref{r:uzycie} (and the choice of $N_\vep$),
for every $N\geq N_\vep$ and $a\in R^P_{n_k}$, we have
\begin{align}
\label{eq:dir^2}
\left|\rho^P(N;n_k,a)-\rho^P(n_k,a)\frac{N}{n_k}\right|<\rho^P(n_k)\leq n_k\leq \frac{\vep}{8}\frac{N}{n_k}.
\end{align}

From now on, we write that an integer number $v$ belongs to $R^P_{n_k}$ if there exists $0\leq v'<n_{k}$ such that $v'=v\!\mod n_k$ and $v'\in R^P_{n_k}$.
We will show that for all $N\geq N_\vep$ and $r\in\Z$, we have
\begin{equation}\label{eq:s-a}
\Big|\frac{1}{N}\sum_{m\leq N}F(S^{P(m)+r}x)-\frac{1}{n_{k}}\sum_{\substack{0\leq a<n_k \\a-r\in R^P_{n_k}\\x_k(a)\neq ?}}{\rho^P(n_k,a-r)}F(S^{a}x)\Big|\leq \vep\|F\|_{\sup},
\end{equation}
and this implies \eqref{eq:cauch^2}.

Recall that $x_{k}\in(\mathcal{A}\cup\{?\})^\Z$ is an $n_{k}$-periodic sequence (used to construct $x$ at stage $k$). Note that for every $a\in\Z$, we have
\[x_{k}(a)\neq\; ? \Rightarrow x(a+j\cdot n_{k})=x_{k}(a)\text{ for every }j\in\Z.\]
This implies that if  $m\leq N$ and $x_{k}(P(m)+r \!\mod n_{k})\neq \;?$, then
\begin{equation}\label{eq:z2^2}
F(S^{P(m)+r}x)=F(S^{P(m)+r \!\mod n_{k}}x).
\end{equation}
Therefore,
\begin{align*}
\#\{&m\leq N: x_{k}(P(m)+r\!\mod n_{k})=\;?\}\\
&=
 \sum_{\substack{0\leq a<n_k\\
a-r\in R^P_{n_k}\\
x_k(a)=?}}\#\{m\leq N: P(m)=a-r\!\mod n_{k}\}
=
\sum_{\substack{0\leq a<n_k\\
a-r\in R^P_{n_k}\\
x_k(a)=?}}\rho^P(N;n_k,a-r).
\end{align*}
Assume that $N\geq N_\vep$. By \eqref{eq:dir^2},  for every integer $v\in R^P_{n_k}$,  we have
\[\rho^P(N;n_k,v)\leq  2{\rho^P(n_k)}\frac{N}{n_k}.\]
In view of \eqref{eq:?}, it follows that
\begin{align*}
\#&\{m\leq N: x_{k}(P(m)+r\!\mod n_{k})=\;?\}\\
&\leq \#\{0\leq a<n_k:a-r\in R^P_{n_k},x_k(a)=?\}  2{\rho^P(n_k)}\frac{N}{n_k}
\\
&\leq 2?_k  {\rho^P(n_k)}\frac{N}{n_k}\leq \frac{\vep}{4}N.
\end{align*}
Let
\[
U_N:=\{m\leq N: x_{k}(P(m)+r\!\mod n_{k})\neq\;?\}.
\]
Then by the above, for every $N\geq N_{\vep}$,
\begin{align}\label{eq:U_N}
\Big|\frac{1}{N}\sum_{m\leq N}F(S^{P(m)+r}x)- \frac{1}{N}\sum_{m\in U_N}F(S^{P(m)+r}x)\Big|\leq \frac{\vep}{4}\|F\|_{\sup}.
\end{align}
But by \eqref{eq:z2^2},
\begin{align*}
\sum_{m\in U_N}F(S^{P(m)+r}x)&=\sum_{\substack{0\leq a<n_{k}\\a-r\in R^P_{n_k}\\
x_k(a)\neq ?}}\sum_{\substack{m\leq N\\P(m)=a-r \!\mod n_{k}}}F(S^{a}x)\\
&=
\sum_{\substack{0\leq a<n_{k}\\a-r\in R^P_{n_k}\\x_k(a)\neq ?}}F(S^{a}x)\#\{m\leq N:\:P(m)=a-r \!\mod n_{k}\}\\
&=
\sum_{\substack{0\leq a<n_{k}\\a-r\in R^P_{n_k}\\x_k(a)\neq ?}}F(S^{a}x)\rho^P(N;n_k,a-r).
\end{align*}
By \eqref{eq:dir^2}, we have
\[\Big|\rho^P(N;n_k,a-r)-\rho^P(n_k,a-r)\frac{N}{n_k}\Big|<\frac{\vep}{8}\frac{N}{n_k}.\]
It follows that
\begin{align*}
&\Big|\frac{1}{N}\sum_{m\in U_N}F(S^{P(m)+r}x)-\frac{1}{n_k}\sum_{\substack{0\leq a<n_{k}\\a-r\in R^P_{n_k}\\x_k(a)\neq ?}}{\rho^P(n_k,a-r)}F(S^{a}x)\Big|\\
&=\Big|\frac{1}{N}\sum_{\substack{0\leq a<n_{k}\\a-r\in R^P_{n_k}\\x_k(a)\neq ?}}F(S^{a}x)\rho^P(N;n_k,a-r)-
\frac{1}{n_k}\sum_{\substack{0\leq a<n_{k}\\a-r\in R^P_{n_k}\\x_k(a)\neq ?}}{\rho^P(n_k,a-r)}F(S^{a}x)\Big|\\
&\leq\frac{1}{N}\sum_{\substack{0\leq a<n_{k}\\a-r\in R^P_{n_k}\\x_k(a)\neq ?}}|F(S^{a}x)|\Big|\rho^P(N;n_k,a-r)-{\rho^P(n_k,a-r)}\frac{N}{n_k}\Big|\\
&\leq\|F\|_{\sup}\frac{\vep}{8}\frac{\#\{0\leq a<n_{k}:x_k(a)\neq ?,a-r\in R^P_{n_k}\}}{n_k}
\leq \|F\|_{\sup}\frac{\vep}{8}.
\end{align*}
Together with \eqref{eq:U_N}, this gives \eqref{eq:s-a}, which completes the proof in the case of $F$ depending only on the zero coordinate.
The rest of the proof runs as in the proof of Theorem~\ref{thm:PNT}, this is by passing to the Toeplitz sequences $x^{(m)}\in(\mathcal{A}^{2m+1})^\Z$ for $m\geq 1$.
\end{proof}

\begin{Remark}\label{r:rpimediv}
Denote by $\PP_{(n_t)}$ the set of all prime divisors of elements of the sequence $(n_t)_{t\geq 1}$.
In view of Corollary~\ref{cor:rhoP}, $?_t=o({p(n_t)}/d^{\omega(n_t)})$ implies \eqref{assum:?}.
Unfortunately, if $\PP_{(n_t)}$ is finite then the sequence $({p(n_t)}/d^{\omega(n_t)})_{t\geq 1}$ is
bounded, so Theorem~\ref{thm:posP}, in the way, is not applicable. Fortunately, if $\PP_{(n_t)}$ is infinite then ${p(n_t)}/d^{\omega(n_t)}\to+\infty$
as $t\to+\infty$, so Theorem~\ref{thm:posP} applies to a non-trivial class of regular Toeplitz shifts, in particular, it applies when the periodic sequences $x_t$ defining $x$ have a bounded number of ``?''.

However, Theorem~\ref{thm:posP} applies to a much wider class of regular Toeplitz shifts when $P(n)=n^2$.
Then, by Corollary~\ref{cor:n^2}, $?_t=o(\sqrt{n_t})$ implies \eqref{assum:?}.
Here the finiteness or the infinity of the set $\PP_{(n_t)}$ does not matter.
\end{Remark}

The assumption \eqref{assum:?}
about the growth of the sequence $(?_t)_{t\geq 1}$ is the least restrictive when all $n_t$ are square-free.
Then, by the second part of Corollary~\ref{cor:n^2}, $?_t=o({n_t}/2^{\omega(n_t)})$ implies \eqref{assum:?}.
Therefore,  $?_t=O({n_t}^{1-\frac{1}{\log_2\log_2 n_t}})$ also implies \eqref{assum:?}. Indeed, it suffices to show that
$2^{\omega(n)}=o({n}^{\frac{1}{\log_2\log_2 n}})$ for square-free numbers $n\to+\infty$. Suppose that $\omega(n)=k$ and
denote by $(p_l)_{l\geq 1}$ the increasing sequence of all prime numbers. Since
\[\ln n\geq \sum_{l=1}^k\ln p_l\geq k\ln k,\]
we have
\[\frac{2^{\omega(n)}}{{n}^{\frac{1}{\log_2\log_2 n}}}=\frac{2^{k}}{{2}^{\frac{\log_2 n}{\log_2\log_2 n}}}\leq \frac{2^k}{2^{\frac{k\log_2 k}{\log_2(k\log_2 k)}}}
=\frac{1}{2^{\frac{k\log_2\log_2 k}{\log_2 k+\log_2\log_2 k}}}.\]
As $\tfrac{k\log_2\log_2 k}{\log_2 k+\log_2\log_2 k}\to+\infty$ when $k\to+\infty$, this gives $2^{\omega(n)}=o({n}^{\frac{1}{\log_2\log_2 n}})$.

%

\subsection{Counter-examples} \label{s:kwadrat}
We will show that there exists a regular Toeplitz sequence $x\in\{0,1\}^\Z$ with the period structure $(n_t)_{t\geq 1}$
satisfying
\begin{equation}\label{eq:perstrsq}
n_{t+1}=k_{t+1}n_k\text{ with }(k_{t+1},n_t)=1, \; n_{t+1}\geq 2^4n_t^2\text{ and }\sum_{p\in\PP_{(n_t)}}\tfrac{1}{p}<+\infty
\end{equation}
and such that
\begin{equation*}\label{eq:fsnt}
\lim_{t\to\infty}\frac1{\sqrt{n_t}}\sum_{0\leq m<\sqrt{n_t}}F(S^{m^2}x)\text{ does not exist,}
\end{equation*}
where $F(y)=(-1)^{y(0)}$.
Let
\[0<\beta:=\frac{1}{16}\prod_{p\in\PP_{(n_t)}}\frac{p-1}{p}.\]
By \eqref{eq:tpsi}, for every $t\geq 1$, we have
\begin{equation}\label{sq2}
\frac{2^{\omega(n_t)}\widetilde\psi(n_t)}{n_t}\geq 8\beta.
\end{equation}
Passing to a subsequence of $(n_t)_{t\geq 1}$ (and remembering that $\widetilde{\psi}(m)\to\infty$ when $m\to\infty$), we can assume that
\[\sum_{t\geq 1}\frac{1}{\widetilde\psi(k_t)}\leq \frac{1}{2}.\]
Set \[\gamma_t:=\sum_{l=1}^t\frac{1}{\widetilde\psi(k_l)} \; \big(\leq \frac{1}{2}\big).\]
At stage $t$, $x$ is approximated  by the infinite concatenation of $x_t[0,n_t-1]\in\{0,1,?\}^{n_t}$ (that is, we see a periodic sequence of $0,1,?$ with period $n_t$). Successive ``?'' will be filled in the next steps of construction of $x$. We require that:
\begin{gather}
\label{sq3}
\{0\leq i< n_t:\: x_t(i)=?\}\subset R_{n_t};\\
\label{sq4}
\#\{a\in \widetilde R_{n_t}:\:x_t(a)=?\}\geq(1-\gamma_t)\widetilde \psi(n_t);\\
\label{sq5}
\#\{0\leq m<\sqrt{n_t}:\:x_t(m^2)=?\}\geq\beta\sqrt{n_t}.
\end{gather}
Recall that, in view of Lemma~\ref{lem:Nna} (remembering that $P^{-1}(n_{t+1})=\sqrt{n_{t+1}}$), \eqref{eq:omegarho} and~\eqref{eq:perstrsq},
for each $a\in \widetilde R_{n_t}$, we have\footnote{$\N^2$ stands for $\{m^2:m\geq 0\}$.}
\begin{align*}
\#(&\{a+jn_t:\:0< j<k_{t+1}\}\cap \N^2)\geq
\#(\{m^2=a~\text{mod }n_t:\: m^2<n_{t+1}\})-1\\
&\geq
\Big(\frac{\sqrt{n_{t+1}}}{n_t}-1\Big)\rho(n_{t},a)-1\geq
\Big(\frac{\sqrt{n_{t+1}}}{n_t}-1\Big)\frac12 2^{\omega(n_t)}-1\\
&\geq
\frac12 2^{\omega(n_t)}\Big(\frac{\sqrt{n_{t+1}}}{n_t}-2\Big)\geq \frac142^{\omega(n_t)}\frac{\sqrt{n_{t+1}}}{n_t},
\end{align*}
so
\begin{equation}\label{sq7}
\#(\{a+jn_t:\:0< j<k_{t+1}\}\cap \N^2)\geq \frac{2^{\omega(n_t)}}{4}\frac{\sqrt{n_{t+1}}}{n_t}.
\end{equation}
By the definition of the sets $R_n$ and $\widetilde R_n$, we have
\begin{align}
\label{sql0}
& R_{n_{t+1}}\subset\bigcup_{a\in R_{n_t}}
\{a+jn_t:\:0\leq j<k_{t+1}\},\\
\label{sql1}
&\widetilde R_{n_{t+1}}\subset\bigcup_{a\in\widetilde R_{n_t}}
\{a+jn_t:\:0\leq j<k_{t+1}\}.
\end{align}
Moreover, by Lemma~\ref{lem:psi}, for every $a\in \widetilde R_{n_t}$, we have
\begin{equation}\label{sql2}
\#\{i\in \widetilde R_{n_{t+1}}:i=a\!\mod n_t\}=\#\widetilde R_{k_{t+1}}=\widetilde \psi(k_{t+1}).
\end{equation}

We need to describe now which and how we fill "?" in
$x_{t+1}[0,n_{t+1}-1]$. This block is divided into $k_{t+1}$ subblocks
\[\underbrace{x_t[0,n_t-1]x_t[0,n_t-1]\ldots x_t[0,n_t-1]}_{k_{t+1}}.\]

We fill in {\bf all} "?" in the first block $x_t[0,n_t-1]$ in such a way to ``destroy'' the convergence of averages in \eqref{eq:perstrsq}
for the time $n_t$, namely
\[
\frac1{\sqrt{n_t}}\sum_{0\leq m<\sqrt{n_t}}F(S^{m^2}x)=
\frac1{\sqrt{n_t}}\Big(\sum_{\substack{m<\sqrt{n_t}\\x_t(m^2)=0}}1-
\sum_{\substack{m<\sqrt{n_t}\\ x_t(m^2)=1}}1+\sum_{\substack{m<\sqrt{n_t}\\x_t(m^2)=?}}(-1)^{x(m^2)}\Big).
\]
And, since the number of $m$ in the last summand is at least $\beta\sqrt{n_t}$ in view of~\eqref{sq5}, we can fill in these places at stage $t+1$ to obtain the sum completely different that the known number which we had from stage $t-1$. We also fill in (in an arbitrary way) the remaining places in $\{0,\ldots, n_{t}-1\}$.

We fill in (in an arbitrary way) all places in $\{n_t,\ldots, n_{t+1}-1\}\setminus  R_{n_{t+1}}$ and only these places, so that
\eqref{sq3} will be satisfied at stage $t+1$.

We must remember that for any $a\in R_{n_t}$ if $x_t(a)\neq ?$ then for every $0\leq j<k_{t+1}$,  we have
$x_{t+1}(a+jn_t)=x_{t}(a+jn_t)=x_t(a)\neq ?$. Moreover,  for any $a\in\widetilde R_{n_t}$ if $x_t(a)= ?$ then for every $0< j<k_{t+1}$
with $a+jn_t\in\widetilde R_{n_{t+1}}$ we have $x_{t+1}(a+jn_t)=?$. In view of \eqref{sql1}, this gives
\begin{align*}
\#\{& i\in\widetilde R_{n_{t+1}}:x_{t+1}(i)\neq ?\}\\
&\leq\widetilde \psi(n_t)+\sum_{a\in\widetilde R_{n_t}:x_t(a)\neq ?}\#\{a+jn_t\in\widetilde R_{n_{t+1}}:0<j<k_{t+1}\}.
\end{align*}
In view of \eqref{sql2} and \eqref{sq4}, it follows that
\begin{align*}
\#&\{ i\in\widetilde R_{n_{t+1}}:x_{t+1}(i)\neq ?\}\leq\widetilde\psi(n_t)+(\widetilde\psi(k_{t+1})-1)\#\{a\in\widetilde R_{n_t}:x_t(a)\neq ?\}\\
&\leq\widetilde \psi(n_t)+(\widetilde \psi(k_{t+1})-1)\gamma_t\widetilde \psi(n_{t})=\Big(\gamma_t+
\frac{1-\gamma_t}{\widetilde\psi(k_{t+1})}\Big)\widetilde\psi(n_{t+1})
\leq \gamma_{t+1}\widetilde\psi(n_{t+1}).
\end{align*}
Therefore, at stage $t+1$, also \eqref{sq4}  is satisfied.

A similar argument combined with \eqref{sq7}, \eqref{sq4} and \eqref{sq2} shows that
\begin{align*}
\#\{ &0\leq m^2<n_{t+1}:x_{t+1}(m^2)= ?\}=
\#\{ i\in R_{n_{t+1}}\cap\N^2:x_{t+1}(i)= ?\}\\
&\geq\sum_{a\in  R_{n_t}:x_t(a)= ?}\#\{a+jn_t\in R_{n_{t+1}}\cap\N^2:0<j<k_{t+1}\}\\
&\geq\sum_{a\in \widetilde R_{n_t}:x_t(a)= ?}\frac{2^{\omega(n_t)}}{4}\frac{\sqrt{n_{t+1}}}{n_t}
=\frac{\sqrt{n_{t+1}}}{4n_t}2^{\omega(n_t)}\#\{a\in\widetilde R_{n_t}:x_t(a)= ?\}\\
&=(1-\gamma_t)\frac{\sqrt{n_{t+1}}}{4n_t}2^{\omega(n_t)}\widetilde\psi(n_t)\geq \beta\sqrt{n_{t+1}}.
\end{align*}
Therefore, at stage $t+1$, also \eqref{sq5}  is satisfied. This completes the construction.

\begin{Remark}
In view of \eqref{sq3}, in the constructed example of Toeplitz system $(X_x,S)$ we have $?_t\leq \psi(n_t)$. Moreover, $\psi(n_t)=o(\varphi(n_t))$.
Indeed, by Proposition~\ref{prop:rhop}, for every prime number $p$ we have $\psi(p^n)\leq p^{n-1}\tfrac{p+2}{2}$. It follows that
\[\frac{\psi(p^n)}{\varphi(p^n)}\leq \frac{1}{2}\cdot\frac{p+2}{p-1}\leq \frac{3}{4}\]
for all prime $p\geq 7$. It follows that
\[\frac{\psi(n_t)}{\varphi(n_t)}=O\Big(\Big(\frac{3}{4}\Big)^{\omega(n_t)}\Big)=o(1).\]
Consequently, we have $?_t=o(\varphi(n_t))$. Therefore, in view of Theorem~\ref{thm:PNT}, $(X_x,S)$ satisfies a PNT.
\end{Remark}

\appendix
\section{The diameter of a tower}\label{app:?}
Let $x\in \mathcal{A}^\Z$ be a Toeplitz sequence with the periodic structure given by $(n_t)_{t\geq 1}$.
Recall that
\[\operatorname{Per}_{n_t}(x)=\{a\in \Z:x(a+jn_t)=x(a)\text{ for all }j\in\Z\}.\]
Let $\operatorname{Aper}_{n_t}(x):=\Z\setminus\operatorname{Per}_{n_t}(x)$. Then, we define the periodic sequence $x_t\in(\mathcal{A}\cup\{?\})^{\Z}$ by:
$x_t(k)=x(k)$ if $k\in \operatorname{Per}_{n_t}(x)$ and $x_t(k)=?$ if $k\in \operatorname{Aper}_{n_t}(x)$.
Note that the density of the set $\operatorname{Aper}_{n_t}(x)$ is equal to $\frac{?_t}{n_t}$, where
\[?_t=\#\{0\leq k<n_t:x_t(k)=?\}=\#(\operatorname{Aper}_{n_t}(x)\cap\{0,1,\ldots,n_t-1\}).\]
It follows that the regularity of $(X_x,S)$ is equivalent to $?_t=o(n_t)$.

\begin{Lemma}\label{lem:?Et}
For any Toeplitz sequence $x\in \mathcal{A}^\Z$ we have
\[?_t\leq \delta(E^t)\leq 3?_t\text{ for every }t\geq 1.\]
\end{Lemma}

\begin{proof}
First note that for every $0\leq j<n_t$ we have
\[E^t_j=\{y\in X_x:y(k-j)=x(k)=x_t(k)\text{ for all }k\in \operatorname{Per}_{n_t}\}.\]
Moreover, if $k\in \operatorname{Aper}_{n_t}(x)$ then we can find $y,z\in E^t_j$, so that $y(k-j)\neq z(k-j)$.
It follows that
\[\operatorname{diam}(E^t_j)=2^{-\inf\{|n|:n\in \operatorname{Aper}_{n_t}(x)-\{j\}\}}.\]
Suppose that
\[\operatorname{Aper}_{n_t}(x)\cap\{0,1,\ldots,n_t-1\}=\{l_1,l_2,\ldots,l_s\}\]
with $1\leq l_1\leq\ldots\leq l_s\leq n_t$ and $s=?_t$. Thus, $\operatorname{diam}(E^t_{l_i})=1$
and if $l_{i-1}< j<l_i$ ($l_0=l_s-n_t$ and $l_{s+1}=l_1+n_t$) then $\operatorname{diam}(E^t_j)=2^{-\min\{j-l_{i-1},l_{i}-j\}}$.
Therefore,
\[\delta(E^t)=\sum_{0\leq j<n_t}\operatorname{diam}(E^t_j)\geq \sum_{i=1}^s\operatorname{diam}(E^t_{l_i})=s\]
and
\begin{align*}
\delta(E^t)&=\sum_{0\leq j<n_t}\operatorname{diam}(E^t_j)=\sum_{i=1}^s\sum_{\frac{l_{i-1}+l_i}{2}\leq j<\frac{l_{i}+l_{i+1}}{2}}
\operatorname{diam}(E^t_j)\\
&=\sum_{i=1}^s\Big(1+\sum_{1\leq j<\frac{l_{i+1}-l_i}{2}}2^{-j}+\sum_{1\leq j\leq\frac{l_{i}-l_{i-1}}{2}}2^{-j}\Big)\leq 3s,
\end{align*}
which completes the proof.
\end{proof}
As the regularity of $x$ is equivalent to $?_t=o(n_t)$, we have the following conclusion.
\begin{Cor}
A Toeplitz sequence is regular if and only if $\delta(E^t)=o(n_t)$.
\end{Cor}

\section{Sturmian dynamical systems satisfy a PNT}\label{s:dosturmu}
Let $T:\T\to\T$ ($\T:=\R/\Z$) be an irrational rotation on $\T$ by $\alpha$. For every non-zero $\beta\in\T$
let $\{A_0,A_1\}$ be the partition given by the intervals $A_0=[0,\beta)$ and $A_1=[\beta,1)$. For every $x\in \T$
denote by $\bar{x}\in\{0,1\}^\Z$ the code of $x$ defined by $\bar{x}(k)=i$ if and only if $T^kx\in A_i$.
Finally, denote by $X_{\alpha,\beta}\subset\{0,1\}^\Z$ the closure of the set
$\{\bar{x}\in\{0,1\}^\Z:x\in\T\}$. Since $X_{\alpha,\beta}$ is an invariant subset for the left shift $S$ on $\{0,1\}^\Z$,
we can focus  the topological dynamical system $S:X_{\alpha,\beta}\to X_{\alpha,\beta}$.

\begin{Th}
For the topological dynamical system $S:X_{\alpha,\beta}\to X_{\alpha,\beta}$ a PNT holds.
\end{Th}

\begin{proof}
For every $y=(y(n))_{n\in\Z}\in X_{\alpha,\beta}$ the set
$\bigcap_{n\in\Z}\overline{A}_{y(n)}\subset\T$ has exactly one element $\pi(y)\in\T$. Moreover, $\pi:X_{\alpha,\beta}\to\T$
is a continuous map intertwining $S$ and $T$ and there exists a unique $S$-invariant probability measure $\mu$ on $X_{\alpha,\beta}$.
The $\pi$-image of $\mu$ coincides with Lebesgue measure on $\T$.

By Vinogradov's theorem, for any character $f(x)=e^{2\pi in x}$, $n\in\Z$, we have
\begin{equation}\label{eq:PNTstong}
\lim_{N\to\infty}\frac1{\pi(N)}\sum_{p<N}f(T^px)=\int_{\T} f(x)\,d x\text{ for every }x\in\T.
\end{equation}
Since every continuous function $f:\T\to\C$
is uniformly approximated by trigonometric polynomials, \eqref{eq:PNTstong} holds also for any continuous $f$.
Moreover, \eqref{eq:PNTstong} holds for any Riemann integrable $f:\T\to\R$. Indeed, for every $\vep>0$ there are two
continuous functions $f_-,f_+:\T\to\R$ such that $f_-(x)\leq f(x)\leq f_+(x)$ for every $x\in\T$ and
$\int_{\T}(f_+(x)-f_-(x))dx<\vep$. It follows that
\begin{align*}
\limsup_{N\to\infty}\frac1{\pi(N)}\sum_{p<N}f(T^px)&\leq \lim_{N\to\infty}\frac1{\pi(N)}\sum_{p<N}f_+(T^px)\\
&=
\int_{\T} f_+(x)\,d x< \int_{\T} f(x)\,d x+\vep
\end{align*}
and
\begin{align*}
\liminf_{N\to\infty}\frac1{\pi(N)}\sum_{p<N}f(T^px)&\geq \lim_{N\to\infty}\frac1{\pi(N)}\sum_{p<N}f_-(T^px)\\
&=
\int_{\T} f_-(x)\,d x> \int_{\T} f(x)\,d x-\vep.
\end{align*}
As $\vep>0$ can be chosen freely, this gives \eqref{eq:PNTstong}.

Suppose that $f:X_{\alpha,\beta}\to\R$ depends only on finitely many coordinates. More precisely, assume that
$f(y)=g(y(-n),\ldots,y(n))$ for some $g:\{0,1\}^{2n+1}\to\R$. Then there exists $F:\T\to\R$ such that $f=F\circ\pi$
and $F$ is constant on the atoms of the partition $\bigvee_{i=-n}^nT^{-i}\{A_0,A_1\}$ (for example, if $n=0$ and $f$ is the characteristic function of $\{y\in X_{\alpha,\beta}:\: y(0)=0\}$ then $F$ is $\mathbf{1}_{A_0}$). It follows that $F$
is Riemann integrable. Therefore, for every $y\in X_{\alpha,\beta}$, we have
\begin{align*}
\frac1{\pi(N)}\sum_{p<N}f(S^py)=\frac1{\pi(N)}\sum_{p<N}F(T^p\pi(y))\to \int_{\T}F(x)dx=\int_{X_{\alpha,\beta}}f\,d\mu.
\end{align*}
Since every continuous function $f:X_{\alpha,\beta}\to\R$ is uniformly approximated by functions depending on finitely many coordinates,
\[\frac1{\pi(N)}\sum_{p<N}f(S^py)\to\int_{X_{\alpha,\beta}}f\,d\mu\text{ for any }y\in X_{\alpha,\beta}
\]
holds for every continuous $f$.
\end{proof}

{
\subsection*{Acknowledgements}
We would like to thank  the
anonymous referee for suggestions to improve the  paper.}


\begin{thebibliography}{99}
\bibitem{Ab-Ka-Le}
  H.\ El Abdalaoui, M. Lema\'nczyk, S.\ Kasjan,
	{\em 0-1 sequences of the Thue-Morse type and Sarnak's conjecture},
	Proc.\ Amer.\ Math.\ Soc.\ {\bf 144} (2016), 161-176.


\bibitem{Bo-Fi-Fi}
 M.~Boyle, D.~Fiebig, U.~Fiebig,
 \emph{Residual entropy, conditional entropy and subshift covers},
 Forum Math.\ \textbf{14} (2002), 713-757.

\bibitem{Bo}
 J.~Bourgain,
 \emph{An approach to pointwise ergodic theorems}, In Geometric aspects of functional analysis
 (1986/87), volume 1317 of Lecture Notes in Math., pages 204-223. Springer, Berlin, 1988.

\bibitem{Bo1}
 J.~Bourgain,
 \emph{M\"obius-Walsh correlation bounds and an estimate of Mauduit and Rivat}, J.\ Anal.\ Math.
 \textbf{119} (2013), 147-163.

\bibitem {Bo2}
 J.~Bourgain,
 \emph{On the correlation of the M\"obius function with rank-one systems},
 J.\ Anal.\ Math.\ \textbf{120} (2013), 105–130.

\bibitem{Dir}
  P.G.L.~Dirichlet,
	\emph{Lectures on number theory. Supplements by R. Dedekind.}
  Translated from the 1863 German original and with an introduction by John Stillwell. History of Mathematics, 16. American Mathematical Society,    Providence, RI; London Mathematical Society, London, 1999. xx+275 pp.
	
\bibitem{Do}
 T.~Downarowicz,
 \emph{Survey of odometers and Toeplitz flows}.
 Algebraic and topological dynamics, 7–37, Contemp.\ Math., 385, Amer.\ Math.\ Soc., Providence, RI, 2005.

\bibitem{Fe-Ku-Le}
 S.~Ferenczi, J.~Ku\l aga-Przymus, M.~Lema\'nczyk,
 \emph{Sarnak's Conjecture -- what's new}, in:
 Ergodic Theory and Dynamical Systems in their Interactions with Arithmetics and Combinatorics,  CIRM Jean-Morlet Chair, Fall 2016,
 Editors: S.~Ferenczi, J.~Ku\l aga-Przymus, M.~Lema\'nczyk,
    Lecture Notes in Mathematics 2213,  Springer International Publishing, pp.\ 418.
		
\bibitem{Fe-Ma}
 S.~Ferenczi, C.~Mauduit,
 \emph{On Sarnak’s conjecture and Veech’s question for interval exchanges},
 J.\ Anal.\ Math.\ \textbf{134} (2018), 545–573.

\bibitem{Gr}
 B.~Green,
 \emph{On (not) computing the M\"obius function using bounded depth circuits},
 Combin.\ Probab.\ Comput.\ \textbf{21} (2012), 942-951.

\bibitem{Gr-Ta}
 B.~Green, T.~Tao,
 \emph{The M\"obius function is strongly orthogonal to nilsequences},
 Annals of Math.\ (2), \textbf{175} (2012), 541-566.


\bibitem{KLR0}
 A.\ Kanigowski, M.\ Lema\'nczyk, M. Radziwi\l\l,
 {\em Rigidity in dynamics and M\"obius disjointness},
 arXiv:1905.13256 (submitted).

\bibitem{Ka-Le-Ra}
 A.~Kanigowski, M.~Lema\'nczyk, M.~Radziwi\l\l,
 \emph{Prime number theorem for analytic skew products},  
	arXiv:2004.01125 (submitted).

\bibitem{Ka-Le-Ra1}
 A.~Kanigowski, M.~Lema\'nczyk, M.~Radziwi\l\l,
 \emph{Semiprime number theorem for smooth Anzai skew products}, in preparation.

{
\bibitem{La}
 E.\ Landau, 
 \emph{Handbuch der Lehre von der Verteilung der Primzahlen}.
 2 B\"ande. (German) 2d ed. With an appendix by Paul T.\ Bateman. Chelsea Publishing Co., New York, 1953.}

\bibitem{Ma-Ri}
 C.~Mauduit, J.~Rivat,
 \emph{Prime numbers along Rudin–Shapiro sequences},
 J.\ Eur.\ Math.\ Soc.\ \textbf{17} (2015), 2595–2642.

\bibitem{Mu}
 C.~M\"ullner,
 \emph{Automatic sequences fulfill the Sarnak conjecture},
 Duke Math.\ J. \textbf{166} (2017), 3219-3290.

\bibitem{Nar}
  W.\ Narkiewicz,
  \emph{Number theory}. Translated from the Polish by S.\ Kanemitsu. World Scientific Publishing Co., Singapore; distributed by Heyden \& Son, Inc., Philadelphia, PA, 1983. xii+371 pp.
	
\bibitem{Pa}
R.~Pavlov,
 \emph{Some counterexamples in topological dynamics},
 Ergodic Theory Dynam.\ Systems \textbf{28} (2008), 1291-1322.

\bibitem{Sa}
 P.~Sarnak,
 {\em Three lectures on the Möbius function, randomness and dynamics},
 http://publications.ias.edu/sarnak/.

\bibitem{Sa-talk} 
 P.\ Sarnak,
 {\em M\"obius randomness and Dynamics six years later} at CIRM at 1h 08 minute
 https://library.cirm-math.fr/Record.htm?idlist=1\&record=19282918124910001909

{
\bibitem{Sel}
 A.\ Selberg,
 \emph{An elementary proof of the prime-number theorem for arithmetic progressions},
 Canad.\ J. Math.\ \textbf{2} (1950), 66-78.}


\bibitem{Vi}
I.M.~Vinogradov,
\emph{The method of trigonometrical sums in the theory of numbers},
(Russian) Trav.\ Inst.\ Math.\ Stekloff \textbf{23}, (1947). 109 pp.

\bibitem{Wi}
M.~Wierdl,
\emph{Pointwise ergodic theorem along the prime numbers},
Israel J.\ Math.\ \textbf{64} (1989), 315-336.
\end{thebibliography}
\end{document}